\documentclass[10pt]{amsart}
\usepackage{latexsym,amsfonts,amssymb,amsthm,amsmath,mathrsfs,color,amscd,graphicx,fullpage,parskip,hyperref,bm,bbm,dsfont}

\usepackage{xcolor}

\usepackage{comment}
\includecomment{uncomment}
\usepackage{commath,mathtools,setspace}
\usepackage{paralist}
\usepackage{extarrows}

\excludecomment{comment} 

\newcommand{\s}[1]{{\mathcal #1}}
\newcommand{\sr}[1]{{\mathscr #1}}
\newcommand{\bb}[1]{{\mathbb #1}}

\newcommand{\ip}[2]{\left\langle #1,#2 \right\rangle}

\newcommand{\firststep}{\setcounter{step}{1}\textbf{Step \arabic{step}.} }
\newcommand{\nextstep}{\stepcounter{step}\textbf{Step \arabic{step}.} }

\DeclareMathOperator{\argmin}{argmin}

\newcommand{\sP}{{\mathscr P}}

\newcommand{\mres}{\mathbin{\vrule height 1.6ex depth 0pt width
0.13ex\vrule height 0.13ex depth 0pt width 1.3ex}}

\newcommand{\bH}{{\mathbb H}}

\newcommand{\R}{{\mathbb R}}

\DeclareMathOperator{\spt}{spt}
\DeclareMathOperator{\range}{range}
\newcommand{\ds}{\displaystyle}

\usepackage{mathtools}

\newtheorem{theorem}{Theorem} 

\newtheorem{lemma}[theorem]{Lemma}
\newtheorem{proposition}[theorem]{Proposition}

\newtheorem{definition}[theorem]{Definition}

\newtheorem{remark}[theorem]{Remark}

\numberwithin{equation}{section}
\numberwithin{theorem}{section}

\newcounter{step}
\begin{document}
	
	\title[Monotonicity conditions for MFG]{On monotonicity conditions for Mean Field Games}
	
	\author{P.~Jameson Graber}
	\address{J. Graber: Baylor University, Department of Mathematics;\\
		Sid Richardson Building\\
		1410 S.~4th Street\\
		Waco, TX 76706\\
		Tel.: +1-254-710- \\
		Fax: +1-254-710-3569 
	}
	\email{Jameson\_Graber@baylor.edu}
	
	\author{Alp\'ar R. M\'esz\'aros}
	\address{A.R. M\'esz\'aros: Department of Mathematical Sciences, Durham University, Durham DH1 3LE, United Kingdom}
	\email{alpar.r.meszaros@durham.ac.uk} 
	
	\date{\today}   
	
	\begin{abstract}
	In this paper we propose two new monotonicity conditions that could serve as sufficient conditions for uniqueness of Nash equilibria in mean field games. In this study we aim for \emph{unconditional uniqueness} that is independent of the length of the time horizon, the regularity of the starting distribution of the agents, or the regularization effect of a non-degenerate idiosyncratic noise. Through a rich class of simple examples we show that these new conditions are not only in dichotomy with each other, but also with the two widely studied monotonicity conditions in the literature, the Lasry--Lions monotonicity and displacement monotonicity conditions. 
	\end{abstract}
	

	\maketitle
	
	
\section{Introduction}

Mean field games were introduced in the seminal papers \cite{lasry07,huang2006large} to model Nash equilibria among a continuum of players in stochastic or deterministic differential games.
A fundamental mathematical question is whether such equilibria are unique.
Lasry and Lions introduced a convenient criterion, referred to as the \emph{Lasry--Lions monotonicity condition}, that can guarantee uniqueness of equilibria for a large class of mean field games.
Since then, the Lasry--Lions (LL) monotonicity condition has been the most popular criterion used to establish uniqueness in the mathematical literature on mean field games.
Conversely, various non-uniqueness results have been recently obtained in the literature in the absence of LL monotonicity on the data (see \cite{bardi2019on, briani2018stable, cirant2019on}). 
Indeed, the state of the art circa 2019 could be justly represented by the following comment by Cardaliaguet and Porretta in their notes on mean field games \cite[Section 1.3.2]{cardaliaguet2020intro}:
\begin{quotation}
	``Let us mention that there are no general criteria for the uniqueness of solutions to (1.29) [a generic mean field game system] in arbitrary time horizon $T$, except for the Lasry--Lions monotonicity condition\dots''
\end{quotation}

Nevertheless, more recently, several other authors have brought to light another criterion known as \emph{displace\-ment monotonicity} \cite{ahuja2016wellposedness,bensoussan2019stochastic, bensoussan2020control, gangbo2020global,gangbo2021mean,meszaros2021mean}, which is a sufficient condition on the data to ensure uniqueness of Nash equilibria for a general class of games. Importantly, displacement monotonicity and LL monotonicity are in dichotomy, meaning that neither one necessarily implies the other (see \cite{gangbo2020global,gangbo2021mean}). These findings revealed that not only the LL monotonicity condition is not a necessary condition regarding uniqueness issues for mean field games, but there might be other regimes of sufficient conditions on the data which could ensure the uniqueness of Nash equilibria.

Given this brief history, a natural question arises:
\begin{quotation}
	Are there other monotonicity conditions under which uniqueness of the mean field equilibrium is guaranteed?
\end{quotation}
The main purpose of this note is to provide a strong affirmative answer to this question.
In what follows, we establish two new general monotonicity conditions, in addition to LL and displacement monotonicity, such that all four types of conditions are in dichotomy with one another--none of them implies any of the others.
We do not claim that these four conditions constitute a definitive list.
On the contrary, we conjecture that other general criteria guaranteeing the uniqueness of equilibria wait to be discovered. Throughout the paper we will consider deterministic problems. However, we expect the newly proposed monotonicity conditions to translate naturally to models subject to noise, thought not without additional challenges (see Remark \ref{rem:deterministic}). These issues will be the subject of future research.

In this paper, by ``uniqueness'' we generally mean that the Nash equilibrium is unique for \emph{arbitrary initial measures and arbitrary time horizons}, without any particular help from the regularization effect of idiosyncratic noise. In particular, our focus here will always be on deterministic models. It is well-known that uniqueness tends to hold under more or less arbitrary structural assumptions on the data, so long as the time horizon is sufficiently small (\cite{carmona2018probabilisticII,cirant2020short,gangbo2015existence,mayorga2019short}).
By contrast, in this investigation, we will always insist that the time horizon be arbitrarily large.
We may occasionally restrict our attention to only certain classes of initial measures--for example, those having an upper bound on one or more of their moments--but generally speaking, ``uniqueness'' is taken to be synonymous with ``unconditional uniqueness.''
In this setting, it is natural to appeal to monotonicity as a general criterion for uniqueness.
Indeed, mathematical analysis offers virtually no other generic tools to guarantee uniqueness of a fixed point.
Nevertheless, for a sufficiently rich mathematical model, we can always ask, \emph{``Monotone in what sense?''}
This will be the crucial point in our investigation.
It turns out that a mean field game can be construed as a fixed point problem in several alternative parameter spaces, all of which involve radically different sensitivity analyses.
By turning our attention to each parameter separately, we derive new monotonicity conditions to establish uniqueness.

\subsection*{Our contributions in this paper}

The main results of this paper are as follows.
In Section \ref{sec:4 conditions}, we state four types of monotonicity conditions on the data, two of which, labeled \eqref{Sigma} and \eqref{L2L2}, are new. The main idea behind the condition \eqref{Sigma} is as follows. If we assume that the measure dependence of the data (the Lagrangian and final cost functions) has a specific factorization, we can rewrite the classical fixed point formulation of the mean field game, using the new structure imposed by this factorization. This will then lead to a new fixed point formulation of the game \emph{not} in the space of probability measures, but rather in new parameter spaces given by this factorization. This idea was initiated in our parallel work \cite{graber2022conservation}, where our motivation was to find new quantities that are transported through the flow of the feedback strategies in the corresponding mean field games, which in particular gave a new perspective in the study of the associated master equations. In specific situations, the condition \eqref{Sigma} reduces to \eqref{sigma}. This reduction happens when the Lagrangian function does not depend on the measure variable, and so the mean field interaction in the game is though the final data only.
Condition \eqref{L2L2} was initially inspired by \eqref{Sigma}, but in fact it can be seen as more natural: rather than factorizing through an arbitrary parameter space, we replace measures with representative $L^2$ random variables, on which monotonicity has a clear interpretation thanks to the inner product.

The main results of Section \ref{sec:uniqueness} are two-fold. On the positive side, we establish that both of our newly proposed monotonicity conditions can be used to prove uniqueness under suitable hypotheses on the Hamiltonian and final datum. In the same time, we revisit the sufficiency of the LL and displacement monotonicity conditions in connection with the question of uniqueness of Nash equilibria. On the negative side, we point to counterexamples showing that the Lasry--Lions monotonicity condition and \eqref{sigma},  in general do not necessarily provide uniqueness of Nash equilibria in mean field games.

Let us give some comments regarding the philosophy behind these negative results. First, the Lasry--Lions monotonicity condition seems to be inspired from a PDE analysis perspective, i.e.~characterizing the Nash equilibrium via the solution to the coupled system of Hamilton--Jacobi and Fokker--Planck equations. More precisely, its essence relies in a procedure of ``cross-multiplying the variables and integrating by parts''. This procedure can be carried out successfully, as long as we have enough regularity to justify the formal computations. This is typically the case in the presence of the regularizing effect of a non-degenerate idiosyncratic noise. However, as long as the solutions are too weak (which is the case for instance when the intensity of the noise is degenerate), such a method breaks down, and there is a good reason to believe that uniqueness might fail. Instead of the presence of a non-degenerate idiosyncratic noise, an alternative framework which yields the sufficiency of the LL monotonicity condition for uniqueness of MFG Nash equilibria (as explained in \cite[Section 3.4]{carmona2018probabilistic}) is the uniqueness of optimizers in optimal control problems for representative agents. Such a condition (independently of the length of the time horizon, and in case when the dynamics of the representative agents are given by a linear control system) can be morally guaranteed only in the case of fully convex problems, i.e. when the Lagrangian is jointly convex in the position and velocity variables, while the final datum is convex in the position variable (cf. Theorem \ref{thm:nonunique optimizers}). However, such convexity assumptions together with the LL monotonicity will enforce the final datum to become displacement monotone (see the discussion in \cite[Remark 2.8]{gangbo2021mean}). It is worth mentioning that displacement monotonicity of non-separable Lagrangians implies in particular that the Lagrangian is jointly convex in the position and velocity variables (see \cite[Lemma 2.5]{meszaros2021mean}). Therefore the uniqueness of optimizers in the optimal control problems for representative agents is philosophically more related to the displacement monotonicity of the Hamiltonian and the final datum.

A general phenomenon which underlies both the LL monotonicity and \eqref{Sigma} is that these monotonicity conditions typically ensure uniqueness of the optimal value of the game, i.e.~the value function; however, uniqueness of Nash equilibria (i.e.~the flow of measures describing the distributions of the agents) in general does not follow from uniqueness of the value function, unless the cost functional is \emph{injective on the set of Nash equilibria}. Such non-injectivity properties pose indeed great issues for instance in the case when the value function fails to be differentiable, and so there are potentially multiple optimal feedback strategies. From the PDE perspective, it is again at such scenarios when the above described method fails. Interestingly, both the displacement monotonicity and the condition \eqref{L2L2} prevent such non-injectivity issues, and if they are present, the uniqueness of Nash equilibria is a generic property. Again, philosophically, the LL monotonicity condition is intimately linked to PDE arguments, the displacement monotonicity condition is strongly connected to optimal control arguments, while the monotonicity conditions \eqref{Sigma} and \eqref{L2L2} are inspired from game theoretic arguments.

The next natural question which arises in this study is whether there are any possible connections between all these monotonicity conditions. For our main results in Section \ref{sec:no implications}, we show that all four monotonicity conditions are completely distinct: there does not exist any necessary implication from one to another (although we give some examples for which two or more of these monotonicity conditions may hold at the same time).
Taken together, these results establish that the uniqueness question in mean field games is significantly richer than has been previously shown in the literature, and we believe they provide a starting point for new avenues of research.

\section{The setting} \label{sec:setting}

Let us introduce first some notations and preliminaries that will be used throughout the paper. 
	If $\s{X}$ and $\s{Y}$ are topological spaces, then for a Borel measurable map $f:\mathcal{X} \to \mathcal{Y}$ and a measure $m$ supported on $\s{X}$, we denote by $f_\sharp m$ the push-forward measure supported on $\s{Y}$ given by the relation $(f_\sharp m)(E) := m\del{f^{-1}(E)}$.	
	For $p\ge 1$, we use the notation $\sP_p(\R^d)$ to denote the space of nonnegative Borel probability measures supported in $\R^d$ with finite $p-$order moments. On $\sP_p(\R^d)$ we define the standard $p-$Wasserstein distance $W_p:\sP_p(\R^d)\times\sP_p(\R^d)\to[0,+\infty)$ as,
	$$W_p(\mu,\nu):=\inf\left\{\int_{\R^d\times\R^d}|x-y|^p\dif{\gamma}:\ \ \gamma\in\sP_p(\R^d\times\R^d),(p^1)_\sharp\gamma=\mu,(p^1)_\sharp\gamma=\mu\right\}^{\frac1p}.$$
Classical results imply (cf. \cite{ambrosio2008gradient}) that there exists at least one optimizer $\gamma$ in the previous problem. We denote by $\Pi_o(\mu,\nu)$ the set of all optimal plans $\gamma$.
	
	Let $(\Omega,\s{A},\bb{P})$ be an atomless probability space. We use the notation $\bb{H}:=L^2(\Omega;\R^d).$ It is a well-known result that if $\bb{P}$ has no atoms, then for each $m \in \sP_2(\bb{R}^d)$ there exists $X \in \bb{H}$ such that $X_\sharp \bb{P} = m$.  In this case, $m$ is the law of the random variable $X$ and we use the notation $m=\mathcal{L}_X$.
	
	Using the terminology from \cite{ambrosio2008gradient}(see also \cite[Chapter 5]{carmona2018probabilistic}), we say that a function $U:\sP_2(\R^d)\to\R$ has a {\it Wasserstein gradient} at $m\in\sP_2(\R^d)$, if there exists $D_m U(m,\cdot)\in \overline{\nabla C_c^\infty(\R^d)}^{L^2_m}$ (the closure of gradients of $C_c^\infty(\R^d)$ function in $L^2_m(\R^d;\R^d)$) such that  for all $m'\in\sP_2(\R^d)$ in any small neighborhood of $m$ we have the first order Taylor expansion
	$$
	U(m') = U(m) +\iint_{\R^d\times\R^d} D_mU(m,x)\cdot(y-x)\dif \gamma(x,y) + o(W_2(m,m')),\ \forall\gamma\in\Pi_o(m,m').
	$$
We say that $U$ is {\it differentiable on $\sP_2(\R^d)$} if its Wasserstein gradient exists at any point.
	
For $U:\sP_2(\R^d)\to\R$, we can define its `lift' $\tilde U:\bb{H}\to\R$ by $\tilde U(X):=U(X_\sharp\bb{P})$. By the results from \cite{gangbo2019differentiability} and \cite[Chapter 5]{carmona2018probabilistic} (cf. \cite{lions07}), $U$ is differentiable at $m$, if and only if $\tilde U$ is Fr\'echet differentiable at $X$ for any $X\in\bb{H}$, such that $X_\sharp\bb{P}=m$. In this case we can write the decomposition 
$$
D\tilde U(X) = D_m U(m,\cdot)\circ X\ \ {\rm{in}}\ \bb{H}, \ \ \forall X\in\bb{H}:\ X_\sharp\bb{P}=m,
$$
where $D\tilde U(X)\in\bb{H}$ stands for the Fr\'echet derivative of $\tilde U$ at $X$.

\subsection{Mild solutions to a mean field game}\label{subsec:mild}

Let $T>0$ be a given time horizon, let $L:[0,T]\times\R^d\times\R^d\times\sP_2(\R^d)\to\R$ and $G:\R^d\times\sP_2(\R^d)\to\R$ be given continuous functions. Let $m_0\in\sP_2(\R^d)$ be the initial agent distribution. 

Let $\Gamma:={\rm{AC}}(0,T;\R^d)$ be the space of absolutely continuous curves on $[0,T]$ with values in $\R^d$. For $t\in[0,T]$, $e_t:\Gamma\to\R^d$ stands for the evaluation map, i.e. $e_t(\gamma)=\gamma(t)$.	
	
Here we give a Lagrangian formulation of the mean field equilibrium problem (cf.~\cite[Section 3]{benamou2017variational} in the case of potential mean field games with local couplings and \cite{cannarsa2018existence} in the case of mean field games with state constraints). In this informal discussion, the data $L$ and $G$ are supposed to satisfy suitable assumptions, as imposed in the mentioned references.

We define an equilibrium to be a measure 
$$\eta\in \sP_{m_0}(\Gamma):=\left\{\tilde\eta\in\sP(\Gamma):\ (e_0)_\sharp\tilde\eta=m_0\right\},$$ such that it is supported on a set of curves $\bar\Gamma\subseteq\Gamma$ and the functional $J^\eta:\Gamma\to\R$,
$$
J^\eta[\gamma]:=\int_0^T L(s,\gamma_s,\dot\gamma_s,(e_s)_\sharp\eta)\dif s + G(\gamma_T,(e_T)_\sharp\eta)
$$
satisfies	
$$
J^\eta[\bar\gamma]\le J^\eta[\gamma],\ \ \forall\gamma\in\Gamma,\forall\bar\gamma\in\bar\Gamma,\ \gamma(0)=\bar\gamma(0).
$$

We can recast this definition as a fixed point problem.
The set $\sr{P}_{m_0}(\Gamma)$ denotes the set of all $\eta \in \sr{P}(\Gamma)$ such that $(e_0)_\sharp \eta = m_0$.
For any $\eta \in \sr{P}_{m_0}(\Gamma)$, there exists a unique Borel measurable family of probabilities $\{\eta_x\}_{x \in \bb{R}^d}$ on $\Gamma$ that disintegrates $\eta$ in the sense that
\begin{equation}
	\begin{cases}
		\eta(\dif \gamma) = \int_{\bb{R}^d} \eta_x(\dif \gamma)\dif m_0(x),\\
		\spt(\eta_x) \subset \Gamma[x] \quad m_0-\text{a.e.} \ x \in \bb{R}^d,
	\end{cases}
\end{equation}
where $\Gamma[x] := \cbr{\gamma \in \Gamma : \gamma(0) = x}$.
We define the set-valued map $E: \sr{P}_{m_0}(\Gamma) \to 2^{\sr{P}_{m_0}(\Gamma)}$ by
\begin{equation}
	E(\eta) = \cbr{\hat{\eta} \in \sr{P}_{m_0}(\Gamma) : \spt(\hat{\eta}_x) \subseteq \Gamma^\eta[x] \quad m_0-\text{a.e.}~x \in \bb{R}^d},
\end{equation}
where
\begin{equation}
	\Gamma^\eta[x] := \cbr{\bar\gamma \in \Gamma[x] : J^\eta[\bar\gamma] \leq J^\eta[\gamma], \ \forall \gamma \in \Gamma[x]}.
\end{equation}
An equilibrium corresponding to the initial measure $m_0$ is simply a fixed point of $E$, i.e.~$\eta \in E(\eta)$.

In this manuscript we exploit additional structural assumptions on the terminal cost $G$ and the Lagrangian, namely that for any given curve $(m_s)_{s\in[0,T]}$ in ${\rm{AC}}([0,T];\sr{P}_2(\bb{R}^d))$
\begin{equation}\label{eq:g = G}
	G(x,m_T) = g(x,\Sigma(m)_T) \quad \text{and} \quad L(s,x,v,m) = \ell(s,x,v,\Sigma(m)_s),  
\end{equation}
for some $\Sigma:{\rm{AC}}([0,T];\sr{P}_2(\bb{R}^d)) \to C([0,T];\mathcal{X})$, and $g:\bb{R}^d \times \mathcal{X} \to \bb{R},\ \ell:[0,T]\times\R^d\times\R^d\times\mathcal{X}\to\R,$ where $\mathcal{X}$ is a given Hilbert space. We emphasize that the space $\mathcal{X}$ has nothing to do with the Hilbert space $\bb{H}$ of $L^2$ random variables, defined above. Such structure will let us investigate new monotonicity properties associated to the mean field game, which will be useful to establish uniqueness of Nash equilibria.

This is quite a natural assumption, as in many application the dependence on the measure variables is through a generalized moment, a finite dimensional projection, etc.  A typical example is for instance when $\mathcal{X}=\bb{R}^k$ is a finite dimensional space, $\Sigma(m)_s = \int_{\R^d} \psi(s,x) \dif m_s(x)$, for $s\in[0,T]$ or $\Sigma(m)_s = \int_{\R^d} (\psi(s,\cdot)\star m_s)(x) \dif m_s(x)$, etc. where $\psi:[0,T]\times\bb{R}^d \to \bb{R}^k$ is given. When $\mathcal{X}$ is finite dimensional, such dependence on the measure variable can be also seen as a  sort of ``dimension reduction'', cf.~\cite{lasry2022dimension, NurSau}.

For $\sigma:[0,T]\to\mathcal{X}$ given, we can now consider the cost
$$
J^{\sigma}[\gamma]:=\int_0^T \ell(s,\gamma_s,\dot\gamma_s,\sigma_s)\dif s + g(\gamma_T,\sigma_T),
$$
where if $\Sigma\left(((e_s)_\sharp\eta)_{s\in[0,T]}\right) = \sigma$, then $J^{\sigma}[\gamma] = J^\eta[\gamma]$.
Likewise, we replace $\Gamma^\eta[x]$ with
\begin{equation*}
	\Gamma^{\sigma}[x] := \cbr{\bar\gamma \in \Gamma[x] : J^{\sigma}[\bar\gamma] \leq J^{\sigma}[\gamma] \ \forall \gamma \in \Gamma[x]}.
\end{equation*}
Finally, we replace $E(\eta)$ with the set-valued map $\mathcal{E}:C([0,T];\mathcal{X}) \to C([0,T];\mathcal{X})$ given by
\begin{equation} \label{eq:Escript}
	\mathcal{E}(\sigma) := \left\{\Sigma\del{((e_s)_\sharp \hat \eta)_{s\in[0,T]}}   : \hat{\eta} \in \sr{P}_{m_0}(\Gamma), \ \operatorname{spt}(\hat{\eta}_x) \subseteq \Gamma^{\sigma}[x] \quad m_0-\text{a.e.}~x \in \bb{R}^d\right\}.
\end{equation}

We consider $\sigma=(\sigma_s)_{s\in[0,T]}$ to be an equilibrium provided that $\sigma \in \mathcal{E}(\sigma)$.
If $\sigma \in C([0,T];\mathcal{X})$ is an equilibrium, then $\sigma = \Sigma\left(((e_s)_\sharp\eta)_{s\in[0,T]}\right)$, for some equilibrium $\eta \in \sr{P}_{m_0}(\Gamma)$.
Conversely, if $\eta$ is an equilibrium in $\sr{P}_{m_0}(\Gamma)$ then $\sigma = \Sigma\left(((e_s)_\sharp \eta)_{s\in[0,T]}\right)$ is an equilibrium in $C([0,T];\mathcal{X})$.

At this point, it is worth noticing that the question of uniqueness of mean field game Nash equilibria could be linked to uniqueness of the fixed point of the operator $\mathcal{E}$ defined over $C([0,T];\mathcal{X})$. In particular, as $\mathcal{X}$ is a Hilbert space, for convenience we embed $C([0,T];\mathcal{X})$ into $L^2([0,T];\mathcal{X})$, which is a Hilbert space on its own, and so the existence and uniqueness of a fixed point will be ensured by a particular monotonicity condition in a Hilbert space. This is what we will investigate later.

\begin{remark} \label{rem:deterministic}
	At this point the reader may notice that it is easier to analyze deterministic games from this point of view.
	For a stochastic game, the set-valued maps $E$ and $\s{E}$ would be defined in terms of a stochastic flow rather than a measure on curves.
\end{remark}

Let us present some particular cases now to demonstrate the richness of phenomena generated by the structural assumption described above. If one assumes that the Lagrangian does not depend on the measure and time variables, i.e. 
\begin{equation}\label{L:reduced}
L(t,x,v,m) = L(x,v),
\end{equation} 
then only the final measure $(e_T)_\sharp \eta$ figures into the cost $J^\eta[\eta]$.

Under the condition \eqref{L:reduced}, in the structural assumption \eqref{eq:g = G} only $\Sigma_T$ appears, and so the cost has the form
$$
J^{\sigma}[\gamma]:=\int_0^T L(\gamma_s,\dot\gamma_s)\dif s + g(\gamma_T,\sigma),
$$
where if $\Sigma_T((e_T)_\sharp\eta) = \sigma$, then $J^{\sigma}[\gamma] = J^\eta[\gamma]$.
$\Gamma^{\sigma}[x]$ will have the form
\begin{equation*}
	\Gamma^{\sigma}[x] := \cbr{\bar\gamma \in \Gamma[x] : J^{\sigma}[\bar\gamma] \leq J^{\sigma}[\gamma] \ \forall \gamma \in \Gamma[x]}.
\end{equation*}
Finally, the operator $\mathcal{E}$ reduces 
to $\mathcal{E}_T:\mathcal{X}\to \mathcal{X}$ given by
\begin{equation} \label{eq:E2}
	\mathcal{E}_T(\sigma) := \left\{\Sigma_T\del{(e_T)_\sharp \hat \eta} : \hat{\eta} \in \sr{P}_{m_0}(\Gamma), \ \operatorname{spt}(\hat{\eta}_x) \subseteq \Gamma^{\sigma}[x] \quad m_0-\text{a.e.}~x \in \bb{R}^d\right\}.
\end{equation}
So, again $\sigma\in\mathcal{X}$ is an equilibrium provided that $\sigma \in \mathcal{E}_T(\sigma)$.
If $\sigma \in \mathcal{X}$ is an equilibrium, then $\sigma = \Sigma_T(m)$ for some equilibrium measure $m \in \sr{P}_2(\bb{R}^d)$, and as before $m = (e_T)_\sharp \eta$ for some equilibrium $\eta \in \sr{P}_{m_0}(\Gamma)$.
Conversely, if $\eta$ is an equilibrium in $\sr{P}_{m_0}(\Gamma)$ then $\sigma = \Sigma_T((e_T)_\sharp \eta)$ is an equilibrium in $\mathcal{X}$. In this case, supposing that $\mathcal{X}$ is a Hilbert space will allow us to to use a monotonicity condition relying on the Hilbert space structure.

We may now notice that under the condition \eqref{L:reduced}, not necessarily imposing special factorization via $\Sigma_T$, we are free to redefine the cost as
$$
J^{m}[\gamma]:=\int_0^T L(\gamma_s,\dot\gamma_s)\dif s + G(\gamma_T,m),
$$
where if $(e_T)_\sharp \eta = m$, then $J^m[\gamma] = J^\eta[\gamma]$.
Likewise, we replace $\Gamma^\eta[x]$ with
\begin{equation*}
	\Gamma^{m}[x] := \cbr{\bar\gamma \in \Gamma[x] : J^{m}[\bar\gamma] \leq J^{m}[\gamma], \ \forall \gamma \in \Gamma[x]}.
\end{equation*}
Finally, we replace $E(\eta)$ with the set-valued map $E_T:\sr{P}_2(\bb{R}^d) \to 2^{\sr{P}_2(\bb{R}^d)}$ given by
\begin{equation} \label{eq:E1}
	E_T(m) := \left\{(e_T)_\sharp \hat \eta : \hat{\eta} \in \sr{P}_{m_0}(\Gamma), \ \spt(\hat{\eta}_x) \subseteq \Gamma^{m}[x] \quad m_0-\text{a.e.}~x \in \bb{R}^d\right\}.
\end{equation}
We say that $m \in \sr{P}_2(\bb{R}^d)$ is an equilibrium in its own right provided $m \in E_T(m)$.
To justify this definition, suppose $m \in E_T(m)$.
Then there exists $\hat \eta \in \sr{P}_{m_0}(\Gamma)$ such that $m = (e_T)_\sharp \hat \eta$ and $\operatorname{spt}(\hat{\eta}_x) \subseteq \Gamma^{m}[x] = \Gamma^{\hat \eta}[x]$ for $m_0$-a.e.~$x$, which implies $\hat \eta$ is an equilibrium.
Conversely, if $\eta$ is an equilibrium, then we observe that $m = (e_T)_\sharp \eta \in E_T(m)$.

We will make use of the following ``lifted version'' of $E_T$, namely $\tilde E_T : \bb{H} \to 2^\bb{H}$ given by
\begin{equation} \label{eq:E1hat}
	\tilde E_T(X) := \{Y \in \bb{H} : \s{L}_Y \in E_T(\s{L}_X)\}.
\end{equation}

In fact, one can study the ``lifted version'' of the operator $E$ itself, namely we can introduce $\tilde E: C([0,T];\bb{H})\to 2^{C([0,T];\bb{H})}$ defined as 
\begin{align}\label{eq:Etilde}
\tilde E(X):=\Big\{& Y=(Y_t)_{t\in[0,T]}\in C([0,T];\bb{H}):\\ 
&\nonumber \s{L}_{Y_t} = (e_t)_\sharp\eta, t\in[0,T],\eta\in\sP_{m_0}(\Gamma), \spt(\eta_x)\subseteq\Gamma^X[x],\ m_0-\text{a.e.}~x \in \bb{R}^d\Big\},
\end{align}
where 

we redefine the cost as
$$
J^{X}[\gamma]:=\int_0^T L(s,\gamma_s,\dot\gamma_s,\mathcal{L}_{X_t})\dif s + G(\gamma_T,\mathcal{L}_{X_T}),
$$
where if $(e_t)_\sharp \eta = \mathcal{L}_{X_t}$, then $J^X[\gamma] = J^\eta[\gamma]$.
Likewise, we replace $\Gamma^\eta[x]$ with
\begin{equation*}
	\Gamma^{X}[x] := \cbr{\bar\gamma \in \Gamma[x] : J^{X}[\bar\gamma] \leq J^{X}[\gamma], \ \forall \gamma \in \Gamma[x]}.
\end{equation*}

We can notice that under the assumption \eqref{L:reduced}, the operator $\tilde E$ reduces in fact to the operator $\tilde E_T.$

\subsection{Four types of monotonicity conditions} \label{sec:4 conditions}

We now introduce four monotonicity conditions, whose implications for uniqueness will be explored in this manuscript. Unless specified otherwise, the data functions $G$ and $L$ are supposed to have sufficient regularity (they are at least continuously differentiable). Precise hypotheses will be assumed on them in the statements of the specific results. First, let us recall the {\it Lasry--Lions} and {\it displacement} monotonicity conditions, studied intensively in the literature.
\begin{definition}[Lasry--Lions monotonicity] \label{def:LL}
	We say that condition (LL) holds for the function $G$ provided that $G$ satisfies the Lasry--Lions monotonicity condition, namely
	\begin{equation} \label{LL} \tag{LL}
		\int_{\R^d} \del{G(x,m_1) - G(x,m_2)} \dif\,(m_1 - m_2)(x) \geq 0 \quad \forall m_1, m_2 \in \sr{P}_2(\bb{R}^d).
	\end{equation}
	We say that condition \eqref{LL} holds \emph{strongly} if the inequality is strict whenever $G(\cdot,m_1)$ and $G(\cdot,m_2)$ are not identical. In other words, equality in \eqref{LL} implies that $G(\cdot,m_1)$ and $G(\cdot,m_2)$ are identical.
\end{definition}
This condition was introduced in the seminal paper \cite{lasry07}, and was the first one in the literature that guaranteed uniqueness of sufficiently regular solutions to MFG systems.
 The LL condition for the function $L$ can be stated in an analogous way.
 	Typically, this is done if $L$ can be ``separated'' as follows:
 	\begin{equation*}
 		L(t,x,v,m) = \ell(t,x,v) + f(t,x,m),
 	\end{equation*}
 	so that Definition \ref{def:LL} applies directly to $f(t,\cdot,\cdot)$.
 	We refer to e.g.~\cite[Section 5, condition 3]{cardaliaguet2017mfgcontrols} and to \cite{kobeissi2022mean} for a more general statement for Lagrangians appearing in so-called mean field games of controls; see also the discussion below in Section \ref{sec:example}.

\begin{remark}\label{rmk:trivial}
We would like to emphasize that the strong version of LL monotonicity in Definition \ref{def:LL} does not rule out the scenario when the function $G$ is independent of the measure variable. Indeed, if $G(x,m) \equiv G(x)$, for all $(x,m)\in\R^d\times\sP_2(\R^d)$, we have that the left hand side of \eqref{LL} is zero for any $m_1,m_2$, but then $G(\cdot,m_1)=G(\cdot,m_2)$ for any $m_1=m_2$. In particular, this means that even the constant zero function is strongly LL monotone. Therefore, we say that the strong LL monotonicity condition is non-trivially satisfied provided there exist $m_1,m_2\in\sP_2(\R^d)$ for which $G(\cdot,m_1)$ and $G(\cdot,m_2)$ are not identical.
\end{remark}

\begin{definition}[Displacement monotonicity] \label{def:D}
	We say that condition \eqref{D} holds for the function $G$ provided that $G$ satisfies the displacement monotonicity condition, namely
	\begin{equation} \label{D} \tag{D}
		\bb{E}\sbr{\del{D_x G(X_1,\s{L}_{X_1}) - D_x G(X_2,\s{L}_{X_2})} \cdot (X_1-X_2)} \geq 0 \quad \forall X_1,X_2 \in \bb{H}.
	\end{equation}
\end{definition}
The condition \eqref{D} naturally extends to Hamiltonians which are not necessarily separated (cf. \cite{gangbo2021mean,meszaros2021mean}).

Displacement monotonicity was first considered in the work \cite{ahuja2016wellposedness} (although under a different name) to study the uniqueness of solutions to MFG with common noise. It became evident later in the works \cite{gangbo2020global,gangbo2021mean,meszaros2021mean} that this condition can serve as an alternative sufficient condition both for the uniqueness of solutions to MFG systems and the well-posedness of the corresponding master equations. The discussions in Subsection \ref{subsec:mild} let us define the following monotonicity conditions.

Assuming the decomposition of $G$ and $L$ via $\Sigma$ as in \eqref{eq:g = G}, we can formulate the following condition.

\begin{definition}[Monotonicity condition in $L^2(0,T;\mathcal{X})$] \label{def:Sigma}
	Suppose that \eqref{eq:g = G} holds and 
	$$\Sigma:{\rm{AC}}([0,T];\sr{P}_2(\bb{R}^d)) \to C([0,T];\mathcal{X})$$ is given. Suppose furthermore that $\mathcal{X}$ is a Hilbert space.
	
	We say that the monotonicity condition \eqref{Sigma} holds provided that, for all initial measures $m_0 \in \sr{P}_2(\bb{R}^d)$ and all time horizons $T > 0$, $I - \mathcal{E}$ is a monotone vector field on $\operatorname{range} \Sigma$, i.e.
	\begin{equation} \label{Sigma} \tag{$\Sigma$}
		\langle\tau_1 - \tau_2, \sigma_1-\sigma_2\rangle_{L^2([0,T];\mathcal{X})} \leq \|\sigma_1 - \sigma_2\|_{L^2([0,T];\mathcal{X})}^2 \quad \forall \sigma_1,\sigma_2 \in \operatorname{range} \Sigma, \ \forall \tau_1 \in \mathcal{E}(\sigma_1), \forall \tau_2 \in \mathcal{E}(\sigma_2).
	\end{equation}
	We say that condition \eqref{Sigma} holds strictly if the inequality is strict for $\sigma_1 \neq \sigma_2$.
\end{definition}

\begin{remark}
We remark that the Hilbert space $L^2([0,T];\mathcal{X})$ appearing in condition \eqref{Sigma} could be replaced by other Hilbert spaces, which could appear naturally in particular problems. We refer to the discussion in Subsection \ref{sec:example}, in the case of mean field games of controls.
\end{remark}

\begin{definition}[Monotonicity condition in $L^2(0,T;\bb{H})$] \label{def:L2L2}
We say that the monotonicity condition \eqref{L2L2} holds provided that, for all initial measures $m_0 \in \sr{P}_2(\bb{R}^d)$ and all time horizons $T > 0$, $I - \tilde E$ is a monotone vector field on $\operatorname{range} \tilde E$, i.e.
\begin{equation} \label{L2L2} \tag{${\rm{L}}^2({\rm{L}}^2)$}
		\int_0^T\bb{E}\sbr{(Y^1_t - Y^2_t) \cdot (X^1_t-X^2_2)}\dif{t} \leq \int_0^T\bb{E}\abs{X^1_t - X^2_t}^2\dif{t},
	\end{equation}
$\forall X^1,X^2 \in C([0,T];\bb{H}), \ \forall Y^1 \in \tilde E(X^1), \forall Y^2 \in \tilde E(X^2).$
	We say that condition \eqref{L2L2} holds strictly if the inequality is strict for $X^1 \neq X^2$.
\end{definition}

\begin{remark}
\begin{enumerate}
\item Let us notice that the choice of the inner product in the definition of \eqref{Sigma} is for convenience. As $L^2([0,T];\mathcal{X})$ is Hilbert space, the terminology of \emph{monotonicity} is used in the standard sense of the word, as for operators between Hilbert spaces. This condition could be rephrased in a different non-Hilbertian setting, but we do not want to deviate the attention of the reader from the main message by further technical constructions.
\item We emphasize that the conditions \eqref{Sigma} and \eqref{L2L2} are significantly different from each other, as we will see later. Philosophically, one might be able to say that \eqref{L2L2} is a particular case of \eqref{Sigma}, by setting $\mathcal{X}=\mathbb{H}$. 
As the operation $\bH\ni X\mapsto \mathcal{L}_X$ is not invertible, one would need to define $\Sigma$ as a multivalued operator,  $\Sigma:{\rm AC}^2([0,T];\sP_2(\R^d))\to 2^{C([0,T];\bH)}$ as $\Sigma(m) = \{X\in C([0,T];\bH):\ \mathcal{L}_{X_t}=m_t,\ \forall t\in[0,T]\}.$
In this case, the definition of $\s{E}(\sigma)$ would also need to be modified accordingly. However, the main difference between the conditions \eqref{L2L2} and \eqref{Sigma} is that the former exploits the `full' probability measure, while the latter one regards only at some specific features of the measure in a given parameter space.
\end{enumerate}  
\end{remark}

It is worth mentioning that if $L$ is independent of the measure and time variables, i.e.~if \eqref{L:reduced} holds, then the previous monotonicity conditions may be replaced with the following ones.

\begin{definition}[Monotonicity condition in $\mathcal{X}$] \label{def:sigma}
	Suppose that \eqref{L:reduced} holds and suppose that $\mathcal{X}$ is a Hilbert space. We say that condition \eqref{sigma} holds provided that, for all initial measures $m_0 \in \sr{P}_2(\bb{R}^d)$ and all time horizons $T > 0$, $I - \mathcal{E}_T$ is a monotone vector field on $\operatorname{range} \Sigma_T$, i.e.
	\begin{equation} \label{sigma} \tag{$\sigma$}
		\langle\tau_1 - \tau_2,\sigma_1-\sigma_2\rangle_\mathcal{X} \leq \|\sigma_1 - \sigma_2\|_{\mathcal{X}}^2 \quad \forall \sigma_1,\sigma_2 \in \operatorname{range} \Sigma_T, \ \forall \tau_1 \in \mathcal{E}_T(\sigma_1), \forall \tau_2 \in \mathcal{E}_{T}(\sigma_2).
	\end{equation}
	We say that condition \eqref{sigma} holds strictly if the inequality is strict for $\sigma_1 \neq \sigma_2$.
\end{definition}

\begin{definition}[Monotonicity condition in $\bb{H}$] \label{def:X}
	Suppose that  \eqref{L:reduced} holds (but \eqref{eq:g = G} does not necessarily).
	We say that condition \eqref{X} holds provided that, for all initial measures $m_0 \in \sr{P}_2(\bb{R}^d)$ and all time horizons $T > 0$, $I - \tilde E_T$ is a monotone vector field on $\bb{H}$, i.e.
	\begin{equation} \label{X} \tag{${\rm{L}}^2$}
		\bb{E}\sbr{(Y_1 - Y_2) \cdot (X_1-X_2)} \leq \bb{E}\abs{X_1 - X_2}^2 \quad \forall X_1,X_2 \in \bb{H}, \ \forall Y_1 \in \tilde E_T(X_1), \forall Y_2 \in \tilde E_T(X_2).
	\end{equation}
	We say that condition \eqref{X} holds strictly if the inequality is strict for $X_1 \neq X_2$.
\end{definition}

Some comments about these definitions are in order.

\begin{remark}
\begin{enumerate}
\item Conditions \eqref{L2L2}, \eqref{Sigma} (and \eqref{X} and \eqref{sigma}) are supposed to hold uniformly with respect to the initial measure $m_0$ and the time horizon $T$.
The reason for this is two-fold.
For one, the conditions \eqref{LL} and \eqref{D} are also uniform with respect to these data.
More importantly, in this manuscript we are concerned with \emph{unconditional} uniqueness, i.e.~uniqueness that does not depend on the time horizon or the initial measure.
\item Monotonicity can go either direction, and we could have easily insisted on the opposite sign in conditions \eqref{L2L2}, \eqref{Sigma} (and \eqref{sigma} and \eqref{X}); we will give the names -\eqref{L2L2}, -\eqref{Sigma} (and -\eqref{X} and -\eqref{sigma}) to the analogous conditions with the inequalities reversed.
Formally, conditions \eqref{LL}, \eqref{D}, \eqref{X}, and \eqref{sigma} all have the ``same sign'', as the example below in Section \ref{sec:sigma does not imply others} illustrates.
Even for conditions \eqref{LL} and \eqref{D} it is known that a form of ``anti-monotonicity'' can also lead to uniqueness; see \cite{mou2022mean}.
Note, however, that for these it is not enough to simply ``reverse the sign'', but one must have sufficiently large anti-monotonicity and impose some specific structural assumptions on the Hamiltonian.
\item It is worth noticing that conditions \eqref{L2L2} and \eqref{Sigma} (and \eqref{X} and \eqref{sigma}) take into account the actual game itself, i.e.~monotonicity is imposed \emph{`after'} we have access to the global-in-time optimal response to an arbitrary crowd trajectory. In contrast to this, the conditions \eqref{LL} and \eqref{D} are imposed in some sense \emph{`locally'}, without having access to the global-in-time optimal response. This is one of the major differences between the philosophy behind our newly proposed conditions compared to the existing ones from the literature.
\end{enumerate}
\end{remark}

In the following sections, we will first address the question of uniqueness of MFG equilibria.
Namely, which of these conditions guarantee that the equilibrium measure is unique?
Second, we address the question of the logical connection between different conditions.
We will see that in general, these conditions are independent and do not imply one another.

\section{Monotonicity conditions and uniqueness} \label{sec:uniqueness}

In this section we investigate the following question: in what cases is there at most one MFG Nash equilibrium, i.e. at most one fixed point of $E$? 

We begin with a simple observation: under \eqref{L:reduced}, if \eqref{X} or -\eqref{X} holds strictly, then uniqueness is immediate!
Indeed, this means there is at most one fixed point $X \in \tilde{E}_T(X)$, and $m = \s{L}_X$ is thereby the unique Nash equilibrium.
By the same argument, if \eqref{L2L2}, -\eqref{L2L2} holds strictly, we again have uniqueness.
Examples where \eqref{X} or -\eqref{X} can be checked will be given below in Section \ref{sec:no implications}.
Examples where \eqref{L2L2} or -\eqref{L2L2} holds could be constructed in a similar spirit as in Section \ref{sec:example}.

Next, we observe that, under 
suitable assumptions on the Lagrangian (which, if it depends on the measure variable, satisfies the corresponding displacement monotonicity condition) \eqref{D} also implies uniqueness.
This follows directly from the results of \cite{meszaros2021mean} (see Theorem 4.5 and Corollary 4.6 in this reference).
It is remarkable that the condition need not be strict.
On the other hand, unlike \eqref{L2L2} (or \eqref{X}), a condition of the form -\eqref{D}, i.e.~with the inequality reverse, certainly does not imply uniqueness.
Indeed, to satisfy -\eqref{D}, it would be sufficient to have a function $G(x,m) = G(x)$ that is \emph{concave} with respect to $x$, i.e.~$D_{xx}^2 G(x) \leq 0$, and independent of $m$. Suppose also that \eqref{L:reduced} holds and $L$ is convex in the $v$ variable.
In that case, the cost is independent of the final measure $m$, which implies that $m$ is an equilibrium whenever it is the push-forward of $m_0$ through an optimal flow at time $T$.
However, since $G$ is concave, there can easily be more than one optimal flow, which implies more than one equilibrium.
See the discussion below for more details.

Suppose that \eqref{L:reduced} is in force. The situation is different and quite interesting for conditions \eqref{LL} and \eqref{sigma}.
In the case of Lasry--Lions monotonicity \eqref{LL}, an inspection of the proof of \cite[Theorem 4.1]{cannarsa2018existence}, for example, implies that when \eqref{LL} holds strictly, then the final cost $G(x,m)$ gives the same value for all equilibria $m$.
However, unless $m \mapsto G(x,m)$ is injective on the set of equilibria, this does not imply uniqueness of the equilibrium.
(Here we are referring only to the final measure; let alone the whole flow of measures!)
Likewise, when \eqref{sigma} holds strictly, we see that the parameter $\Sigma_T(m)$ has the same value for all equilibria $m$.
However, unless $m\mapsto\Sigma_T(m)$ is injective on the set of equilibria, this does not imply uniqueness of the equilibrium.
Thus both \eqref{LL} and \eqref{sigma}, when they hold in a strict (or strong) sense, do imply uniqueness of the value function, but not necessarily of the equilibrium distribution. 

As a consequence, we can conclude in particular that even the strong version of the Lasry--Lions monotonicity does not imply in general the uniqueness of MFG Nash equilibria. This observation seems to be new in the literature. Let us comment more on the literature to date regarding the strong version of the LL monotonicity in connection to the uniqueness of MFG Nash equilibria. As \cite[Theorem 1.4]{cardaliaguet2020intro} states, we can see that when non-degenerate idiosyncratic noise is present, under Lasry--Lions monotonicity we have uniqueness of MFG Nash equilibria if {\it either} the data which depend on the measure variable (the running and the final costs) are strongly LL monotone, {\it or} the Hamiltonian is strongly convex in the momentum variable.

 We consider two examples below. In one of them the final datum is nontrivially strongly LL monotone, while the running cost does not depend on the measure (hence trivially strongly LL monotone, cf. Remark \ref{rmk:trivial}). In the second one the running cost is nontrivially strongly LL monotone, while the final datum is trivially strongly LL monotone. In both cases the Hamiltonian is purely quadratic in the momentum variable, and therefore strongly convex in this variable. These examples
would fit into the setting of \cite[Theorem 1.4]{cardaliaguet2020intro}, except that we do not have noise. Therefore, the message we would like to convey is that regularity of the measure variable (for instance, as a consequence of the regularizing effect of the noise), or some other sufficient assumptions (for instance the uniqueness of optimal feedback strategies in the single agent optimization problems) need to be imposed in addition to the strong monotonicity of the running and final costs or strong convexity of the Hamiltonian, otherwise uniqueness of MFG Nash equilibria might fail. For deterministic problems, such an additional sufficient assumption is that the measure component of the MFG system stays essentially bounded, as we can see in \cite[Theorem 1.8]{cardaliaguet2020intro}. In our examples the measure component is in fact a singular measure.

Under \eqref{L:reduced}, these difficulties can be obviated if one assumes \emph{a priori} that the minimizers of the cost functional $J^m$ are unique for each $m$.
Such an assumption makes it almost inevitable that $x\mapsto G(x,m)$ should be convex (independently of $m$), as the following result shows:
\begin{theorem}\label{thm:nonunique optimizers}
	Let $\phi:\bb{R}^d \to \bb{R}$ be continuous, non-constant, and bounded below, and let it have sub-quadratic growth  at infinity, i.e.~assume there exist $C>0$ and $\alpha \in (0,1)$ such that $\phi(x)\le C(1+|x|^{1+\alpha})$  for all $x\in\R^d.$
	Assume $\phi$ is not convex.
	Then there exists $t^* > 0$ such that for every $t \geq t^*$, there exists $x \in \bb{R}^d$ such that $\R^d\ni y\mapsto\frac{\abs{x-y}^2}{2t} + \phi(y)$ has at least two distinct minimizers.
\end{theorem}
See Section \ref{sec:nonunique optimizers} for the proof.
Theorem \ref{thm:nonunique optimizers} proves that the Cauchy problem for the Hamilton--Jacobi equation
\begin{equation}
	\left\{
	\begin{array}{ll}
	\partial_t u + \frac{1}{2}\abs{\nabla u}^2 = 0, & {\rm{in}}\ (0,+\infty)\times\R^d,\\
	u(0,x) = \phi(x), & x\in \R^d.
	\end{array}
	\right.
\end{equation}
necessarily develops a shock at a certain time $t^* > 0$ depending only on $\phi$, where $\phi$ is \emph{any} non-convex function having sub-quadratic growth at infinity.
Although the result is stated only for quadratic Lagrangians, it underscores how common it is to have more than one minimizer for a given optimal control problem.

In the light of these observations, let us now proceed to state precisely our main results concerning uniqueness.
Our first main result establishes the uniqueness of MFG Nash equilibria under our newly proposed monotonicity condition \eqref{Sigma} (or \eqref{sigma}).
\begin{proposition}\label{prop:uniqueness_sigma}
Let $T>0$ be a given arbitrary large time horizon. Suppose that $G:\R^d\times\sP_2(\R^d)\to\R$ and $L:[0,T]\times\R^d\times\R^d\times\sP_2(\R^d)\to\R$ have the decomposition \eqref{eq:g = G} for some given Hilbert space $\mathcal{X}$, $g:\R^d\times\mathcal{X}\to\R$, $\ell:[0,T]\times\R^d\times\R^d\times\mathcal{X}\to \R$ and $\Sigma:{\rm{AC}}([0,T];\sr{P}_2(\bb{R}^d)) \to C([0,T];\mathcal{X})$. Suppose that the strict version of \eqref{Sigma} holds. Last, suppose that $g$ and $\ell$ are such that the Hamilton--Jacobi equation
\begin{equation}\label{eq:HJsigma}
\left\{
	\begin{array}{ll}
	-\partial_t u + h(t,x,-\nabla_x u,\sigma_s) = 0, & {\rm{in}}\ (0,T)\times\R^d,\\
	u(T,x) = g(x,\sigma_T), & x\in \R^d,
	\end{array}
	\right.
\end{equation}
has a unique classical solution for all $\sigma=(\sigma_s)_{s\in[0,T]}$ in $\range(\Sigma)$. Suppose also that the vector field $D_p h(\cdot,\cdot,-\nabla_x u,\sigma)$ has a globally defined flow. Here $h:[0,T]\times\R^d\times\R^d\times\mathcal{X}\to\R$ is defined in the standard way as $h(t,x,p,\sigma):=\sup_{v\in\R^d}\left\{p\cdot v - \ell(t,x,v,\sigma)\right\}$. Then the corresponding mean field game starting at any $m_0\in\sP_2(\R^d)$ has at most one Nash equilibrium. 

Suppose in addition that $L$ satisfies the reduced form \eqref{L:reduced} and that \eqref{sigma} holds in a strict sense.
Then we can again conclude that the corresponding mean field game has at most one Nash equilibrium.
\end{proposition}

\begin{remark}
The existence of a mean field game Nash equilibrium can be obtain under mild assumptions on the data $L$ and $G$, and this is relatively well documented in the literature. In the particular setting when the problem \eqref{eq:HJsigma} has a regular enough classical solution, this existence question was established in \cite[Theorem 3.6]{meszaros2021mean}. 
\end{remark}

\begin{proof}[Proof of Proposition \ref{prop:uniqueness_sigma}]
Let $\eta$ be an equilibrium (in particular $(e_0)_\sharp\eta = m_0)$, and set $m_s = (e_s)_\sharp \eta$ and $\sigma_s = \Sigma(m)_s$.
Then it follows that the curve $\sigma$ belongs to $\mathcal{E}(\sigma)$, as defined in \eqref{eq:Escript}, i.e.~$\sigma$ is a fixed point of $\s{E}$.
By the strict monotonicity condition \eqref{Sigma}, this fixed point $\sigma$ is unique.
To establish the uniqueness of the measure $\eta$, it is enough to establish that $\eta$ is completely determined by $\sigma$.
Indeed, by our assumptions, we have the uniqueness of characteristics associated to \eqref{eq:HJsigma}, which are also given by the flow of the vector field $D_p h(\cdot,\cdot,-\nabla_x u,\sigma)$, i.e. 
$$
\left\{
\begin{array}{ll}
\partial_s\gamma_s(x) = D_p h(s,\gamma_s(x),-\nabla_x u(s,\gamma_s(x)),\sigma_s), & s\in(0,T),\\
\gamma_0(x) = x, & x\in\R^d.
\end{array}
\right.
$$
It follows that the set of optimal curves is a singleton, i.e.~$\Gamma^\eta[x]$ consists only of the solution to this initial value problem.
Hence the measures $\eta_x$ are Dirac masses determined entirely by $\sigma$, whereupon $\eta$ itself is determined by $\sigma$, as well.

The case where $L$ satisfies \eqref{L:reduced} and \eqref{sigma} is satisfied is entirely analogous and therefore omitted. Under this reduction, in particular $h$ will also be independent of $t$ and $\sigma$.
\end{proof}

\subsection{An example of data satisfying \eqref{Sigma}} \label{sec:example} Let us pause for a moment to present data that fulfill the monotonicity condition \eqref{Sigma}. As we demonstrate below, our proposed monotonicity condition \eqref{Sigma} is well-suited for a rich class of MFG, including so-called MFG of {\it controls}.

In this example, we take both $x$ and $\sigma$ in $\bb{R}^d$, where $d$ is arbitrary. Let $T>0$ be a time horizon, which can be taken to be arbitrarily long.
	Let us assume that $H,L:(0,T)\times\R^d\times\R^d\to \R$ have the standard connection between them, via the Legendre--Fenchel transform, i.e.
	\begin{equation*}
		H = H(t,p,\pi) = \sup_{v\in\R^d} [v\cdot p - L(t,v,\pi)],
	\end{equation*}
	In the case of both $H,L$, the last coordinate $\pi\in\R^d$ will play the role of a special parameter, that we will describe below.
For simplicity, we will assume that the final cost depends only on the position variable $x$, i.e.~$g:\R^d\to\R$.
	
In this MFG of controls, $x$ denotes the position/state of individual agents. $m_0\in\sP(\R^d)$ represents their initial distribution with respect to their positions, while the flow of measures $m:(0,T)\to\sP_2(\R^d)$ is on the {\it velocity} variable (in the Lagrangian coordinate; i.e. on the controls) of the individual agents. Therefore the parameter curve $\pi:[0,T]\to\R^d$, defined as $\pi(t) = \int_{\R^d}v\dif m_t(v)$ represents the barycenter of these measures, and both the Hamiltonian and Lagrangian depend on this finite dimensional quantity.
	
	Given a trajectory $\pi:[0,T]\to\R^d$, we will define $\sigma:[0,T]\to\R^d$  as $\sigma(t) = \int_0^t \pi(s)\dif s$, which satisfies $\sigma(0)=0$.
	Our main assumption in this subsection will be
	\begin{equation} \label{eq:pi monotonicity}
		\pi \mapsto D_p H(t,p,\pi) - \pi \quad \text{is strictly monotone with respect to } \pi.
	\end{equation}
This monotonicity is understood in the classical sense of vector fields defined on finite dimensional spaces, i.e. for any $t,p\in (0,T)\times\R^d$ fixed we impose
\begin{equation} \label{eq:dpH monotone positive}
	\left[D_p H(t,p,\pi_1) - D_p H(t,p,\pi_2)\right]\cdot (\pi_1-\pi_2)> \abs{\pi_1 - \pi_2}^2,\ \ \forall\ \pi_1,\pi_2\in\R^d, \pi_1\neq \pi_2.
\end{equation}
As in our consideration, the direction of the monotonicity can be swapped, we can assume also the opposite sign instead, i.e. 
\begin{equation}\label{eq:dpH monotone negative}
	\left[D_p H(t,p,\pi_1) - D_p H(t,p,\pi_2)\right]\cdot (\pi_1-\pi_2)< \abs{\pi_1 - \pi_2}^2,\ \ \forall\ \pi_1,\pi_2\in\R^d, \pi_1\neq \pi_2.
\end{equation}	
A typical example for the $H$ would be 
\begin{equation}\label{ex:pi}
H(t,p,\pi)= \sup_{v\in\R^d} \cbr{v \cdot p - \ell\del{v + a(t)\pi}} = \ell^*(p) - a(t)p \cdot \pi,
\end{equation}
where $a:[0,T]\to \bb{R}$ is a smooth curve, and where $\ell:\bb{R}^d \to \bb{R}$ is a given function whose Legendre transform is denoted $\ell^*$.
If $a(t) < -1$, then we have the first inequality \eqref{eq:dpH monotone positive}, while if $a(t) > -1$, we have the second \eqref{eq:dpH monotone negative}.

	For a given path $\pi:[0,T]\to\R^d$ and initial condition $z\in\R^d$, the optimal trajectory $x(t;\pi,z)$ is given by solving
	\begin{equation}
		\dot{x}(t) = D_p H(t,-Dg(x(T)),\pi(t)), \quad x(0) = z.
	\end{equation}
We remark that this is because the classical forward-backward Hamiltonian system simplifies as $H$ does not depend on the position variable.
	
We note that we have an MFG Nash equilibrium $m:(0,T) \to \sP_2(\R^d)$ if and only if by defining $\pi(t) = \int_{\R^d}v\dif m_t(v)$, we have that 
$$
m_t= \dot x(t;\pi,\cdot)_\sharp m_0.
$$	
	
	We then define, for a given initial measure $m_0\in\sP_2(\R^d)$, $\Pi:C((0,T);\R^d)\to C((0,T);\R^d)$ as
	\begin{equation} \label{eq:market function}
		\Pi(\pi)_t = \int_{\R^d} \dot{x}(t;\pi,z)\dif m_0(z).
	\end{equation}
	The MFG equilibrium is now related to the fixed point problem 
	$$\pi = \Pi(\pi).$$

	We can now define $\Sigma: \mathcal{Y}\to\mathcal{Y}$ as 
	\begin{equation*}
		\Sigma(\sigma)_t = \int_{\R^d} x(t;\dot{\sigma},z)\dif m_0(z) - \int_{\R^d}z \dif m_0(z),
	\end{equation*}
	where $\mathcal{Y}:=\left\{\gamma\in C^1((0,T);\R^d):\ \gamma(0)=0\right\}$. The subtraction of the constant $ \int_{\R^d}z \dif m_0(z)$ in the definition of $\Sigma$ is performed in order to ensure that $\Sigma$ maps $\mathcal{Y}$ into itself.
	We realize that $\Pi(\pi)_t = \od{}{t}\Sigma(\sigma)_t$.
	
	To keep this exposition brief, let us assume that $g$ is linear, so that $Dg(x) = c$ is a constant in $\bb{R}^d$; this will allow us to quickly arrive at a result for this class of examples.
	Then \eqref{eq:market function} becomes
	\begin{equation*}
		\Pi(\pi)_t = \int D_p H(t,-c,\pi(t))\dif m_0(z).
	\end{equation*}
	We notice that by \eqref{eq:pi monotonicity},  $\Pi(\pi) - \pi$ is monotone in the sense that
	\begin{equation*}
		\ip{\Pi(\pi_1)_t - \pi_1(t) - \del{\Pi(\pi_2)_t - \pi_2(t)}}{\pi_1(t) - \pi_2(t)} > 0\ \  (\text{or } < 0) \quad \forall t
	\end{equation*}
	hence
	\begin{equation*}
		\int_0^T \ip{\Pi(\pi_1)_t - \pi_1(t) - \del{\Pi(\pi_2)_t - \pi_2(t)}}{\pi_1(t) - \pi_2(t)}\dif t > 0\ \  (\text{or } < 0).
	\end{equation*}
	As a consequence of this, we have that $\Sigma$ is monotone on $\mathcal{Y}$, in the sense of the $H^1([0,T])$ inner product for trajectories beginning at $\sigma(0) = 0$. This corresponds exactly to the monotonicity condition \eqref{Sigma} (where the parameter space $\mathcal{X}=\R^d$ and we have changed the inner product from the $L^2$ inner product to an $H^1$ inner product).
	
	The economic interpretation of \eqref{eq:pi monotonicity} is readily available.
	Consider the case of competitive production of a given resource, cf.~\cite{chan2017fracking,gueant2011mean}.
	The control can be the quantity sold at a given moment, i.e.~the 
	derivative of the stock (state).
	Then $\pi$ would represent the total quantity on offer by the market.
	Condition \eqref{eq:pi monotonicity} simply means that one must sell more if the market has more to offer.
	Assumptions under which such a condition holds can be found, for example, by looking at \cite[Lemma 4.1]{graber2022master} for the scalar case; for the vector case, it is enough to assume it holds in each coordinate.
	
	Let us compare condition \eqref{eq:pi monotonicity}, which leads immediately to the monotonicity \ref{Sigma} in our sense, with the Lasry--Lions type monotonicity studied in \cite{cardaliaguet2017mfgcontrols,kobeissi2022mean}.
	Under our restrictions, this condition would become
	\begin{equation} \label{eq:LL for MFGC}
		\int_{\R^d} \del{L\del{t,v,\int_{\R^d} \tilde{v}\dif \mu_1(\tilde{v})} - L\del{t,v,\int_{\R^d} \tilde{v}\dif \mu_2(\tilde{v})}}\dif\, (\mu_1 - \mu_2)(v) \geq 0
	\end{equation}
	for all measures $\mu_1,\mu_2\in\sP_2(\R^d)$. 
	Equation \eqref{eq:LL for MFGC} is not in general equivalent to \eqref{eq:pi monotonicity}. For instance, in the case of the example \eqref{ex:pi}, we have that 
	$$
	L(t,v,\pi)= \ell\del{v + a(t)\pi},
	$$
from which we get 	
\begin{equation} \label{eq:LL MFGC expression}
\begin{split}
	&\int_{\R^d} \del{L\del{t,v,\int_{\R^d} \tilde{v}\dif \mu_1(\tilde{v})} - L\del{t,v,\int_{\R^d} \tilde{v}\dif \mu_2(\tilde{v})}}\dif\, (\mu_1 - \mu_2)(v)\\
&=\int_{\R^d} \del{\ell\del{v + a(t)\int_{\R^d} \tilde{v}\dif \mu_1(\tilde{v})} - \ell\del{v + a(t)\int_{\R^d} \tilde{v}\dif \mu_2(\tilde{v})}} \dif\, (\mu_1 - \mu_2)(v)
\end{split}\end{equation}

\medskip

{\it Claim.} In general, the expression \eqref{eq:LL MFGC expression} will not always have the same sign for all probability measures $\mu_1$ and $\mu_2$.

{\it Proof of Claim.} To see this, let us assume for simplicity the dimension $d = 1$, $a(t) = 1$ is constant, and let us restrict our attention to measures with $\int_{\R} \tilde{v}\dif \mu_1(\tilde{v}) = 1$ and $\int_{\R} \tilde{v}\dif \mu_2(\tilde{v}) = 0$.
Then the expression reduces to
\begin{equation} \label{eq:LL MFGC expression0}
	\int_{\R} \del{\ell(v + 1) - \ell(v)} \dif\, (\mu_1 - \mu_2)(v).
\end{equation}
We take, for example, $\ell(v) = v^4$.
The expression \eqref{eq:LL MFGC expression0} reduces to
\begin{equation} \label{eq:LL MFGC expression1}
	\int_{\R} \del{4v^3 + 6v^2 + 4v + 1} \dif\, (\mu_1 - \mu_2)(v) = \int_{\R} \del{4v^3 + 6v^2} \dif\, (\mu_1 - \mu_2)(v) + 4.
\end{equation}
We can simplify even further by taking $\mu_2 = \delta_0$, the Dirac mass at zero, so that \eqref{eq:LL MFGC expression1} becomes
\begin{equation} \label{eq:LL MFGC expression2}
	\int_{\R} \del{4v^3 + 6v^2} \dif\mu_1(v) + 4.
\end{equation}
It suffices to show that this may be positive or negative, depending on the choice of $\mu_1$.
Take
\begin{equation}
	\mu_1 = b\delta_x + (1-b)\delta_{\frac{1-bx}{1-b}},
\end{equation}
where $b \in (0,1)$ and $x$ is some real number.
Note that $\int v \dif \mu_1(v) = 1$, as desired.
We also have
\begin{equation}
	\int_{\R} \del{4v^3 + 6v^2} \dif\mu_1(v) = (4x^3 + 6x^2)b + \del{4\del{\frac{1-bx}{1-b}}^3 + 6\del{\frac{1-bx}{1-b}}^2}(1-b),
\end{equation}
which, so long as $b \neq 1/2$, is a cubic polynomial in $x$ whose leading term is computed to be
\begin{equation*}
	4bx^3 - \frac{{4b}^3}{(1-b)^2}x^3 = \frac{{4b}(1-2b)}{(1-b)^2}x^3.
\end{equation*}
As a cubic polynomial can take any value in $\bb{R}$, we deduce that \eqref{eq:LL MFGC expression2} may be positive or negative, as desired.

\medskip
	
It turns out that the LL condition holds in this case for purely quadratic functions, i.e.~ $\ell(v) = \abs{v}^2$. 
In this case, by subtracting the squares and making a cancellation, one finds that the expression \eqref{eq:LL MFGC expression} reduces to
\begin{equation}
	2a(t)\abs{\int_{\R^d} v \dif \del{\mu_1 - \mu_2}(v)}^2,
\end{equation}
which can have a definite sign if $a(t)$ is chosen to be positive or negative.
Because of this fact, one can use the standard PDE approach \`a la Lasry--Lions to prove uniqueness in the study of exhaustible resource models for which the demand schedule is linear and hence the Hamiltonian is purely quadratic, cf.~\cite{graber2022master,graber2021nonlocal} (see also \cite{kobeissi2022mean,graber2021weak,bonnans2019schauder}, where the Lagrangian has a different structure allowing LL monotonicity to hold).
It was noticed in \cite{graber2022master,graber2021nonlocal} that the case of a nonlinear demand schedule makes it considerably more difficult to prove uniqueness, and this was done only under a certain smallness assumption.
Although in the present article we do not wish to address the many technicalities that arise in such models, we believe that in future work the approach proposed here could help generalize those uniqueness results.

\begin{remark}

 Recently, in \cite{mou2022mean-2} the authors introduced monotonicity conditions of displacement type (``Assumption 5.1'') for Hamiltonians appearing in mean field games of controls. In principle, one could check whether our condition \eqref{eq:pi monotonicity} implies Assumption 5.1 in that work. However, this would require introducing a large amount of additional notation and making a number of quite delicate calculations. In our context, one would first need to verify that the fixed point problem
	\begin{equation*}
		\zeta = D_p H\del{t,-\eta,\bb{E}[\zeta]}
	\end{equation*}
	has a unique fixed point in the space of $L^2$ random variables, given an arbitrary $L^2$ random variable $\eta$.
	One can show this much by using condition \eqref{eq:pi monotonicity}.
	After this one would need to define $\zeta$ implicitly as a function of $\eta$ and analyze its derivative in order to check that \cite[Assumption 5.1]{mou2022mean-2} holds; this would require additional smoothness assumptions on $H$ and a lot of subtle implicit differentiation (unless $L$ is, say, quadratic with respect to the velocity).
	Although such a calculation is potentially illuminating, we have chosen not to include it for the sake of brevity.
\end{remark}

\medskip

For the rest of the paper, in order to present the richness of phenomena behind the monotonicity conditions in a simple way, we will assume that the Lagrangian is independent of the time and measure variables, i.e. \eqref{L:reduced} holds. In this context, the monotonicity condition \eqref{Sigma} will simply reduce to the condition \eqref{sigma} and \eqref{L2L2} reduces to \eqref{X}. Furthermore, we restrict our attention to the case $\mathcal{X}=\R^k$, and most often we set $k=1$.

\subsection{Strong Lasry--Lions monotonicity or the strict \eqref{sigma} in general do not imply uniqueness of Nash equilibria}
We now give a family of examples for which \eqref{LL} holds, but uniqueness of MFG Nash equilibria does not hold.
\begin{proposition} \label{pr:LL not unique}
	Let $\phi:\R^d\to\R$ be any positive, bounded, function for which there exists some $x_0$ such that \mbox{$\phi(x_0 + y) = \phi(x_0 - y)$} for all $y$.
Assume, moreover, that \mbox{$\phi(x_0 + z) < \phi(x_0)$} for some $z$.
Define the data as follows:
\begin{equation}
	G(x,m) = \phi(x)\sigma_T(m), \quad \sigma_T(m) = \int_{\R^d} \phi \dif m, \quad L(x,v) = \frac{1}{2}\abs{v}^2.
\end{equation}
Then \eqref{LL} is satisfied strongly for $G$, but uniqueness of the MFG equilibria does not hold for the initial measure $m_0 = \delta_{x_0}$.
\end{proposition}
\begin{proof}
	To see that \eqref{LL} is satisfied strongly, notice that
\begin{equation}
	\int_{\R^d} \del{G(x,m_1) - G(x,m_2)}\dif\, (m_1-m_2)(x) = \del{\int_{\R^d} \phi(x)\dif\, (m_1-m_2)(x)}^2.
\end{equation}
For the initial measure $m_0 = \delta_{x_0}$, the mean field game boils down to finding measures $m$ such that
\begin{equation}
	\operatorname{spt}(m) \subset \argmin \cbr{\frac{\abs{x_0-y}^2}{2T} + \phi(y)\int_{\R^d} \phi \dif m : y \in \bb{R}^d}.
\end{equation}
Now with $z$ as above, there exists $\tau$ large enough such that $\frac{\abs{z}^2}{2\tau} + \phi(x_0 + z) < \phi(x_0)$.
It follows that we can find some $y^*$ that minimizes $\frac{\abs{y}^2}{2\tau} + \phi(x_0 + y)$, hence so does $-y^*$ because $\phi(x_0 + y^*) = \phi(x_0 - y^*)$. We emphasize that $y^*\neq 0$, since we have $\frac{\abs{z}^2}{2\tau} + \phi(x_0 + z) < \phi(x_0)$.
Set $T = \frac{\tau}{\phi(x_0 + y^*)}$.
Then 
\begin{equation}
	\pm y^* \in \argmin \cbr{\frac{\abs{x_0-y}^2}{2T} + \phi(y)\phi(x_0 \pm y^*) : y \in \bb{R}^d}.
\end{equation}
It follows that $m = \delta_{\pm y^*}$ are both Nash equilibria.
\end{proof}

	The example given in Proposition \ref{pr:LL not unique} is a game where the measure-dependence occurs only in the final cost.
	However, the basic idea holds even when the running cost depends on the measure, as well.
	We give a simple example to illustrate the general idea.
	\begin{proposition}
		\label{pr:LL really not unique}
		Let $\phi:\R^d\to\R$ be given by $\phi(x) = \sqrt{2\abs{x}}$ and set $F(x,m) = \phi(x)\int_{\bb{R}^d} \phi \dif m$.
		Fix a final time $T$.
		Set the final coupling $G(x,m) = -T\abs{x}$, and set the Lagrangian to be
		\begin{equation}
			L(x,v,m) = \frac{1}{2}\abs{v}^2 + F(x,m) = \frac{1}{2}\abs{v}^2 + \phi(x)\int_{\bb{R}^d} \phi \dif m.
		\end{equation}
		Then \eqref{LL} is satisfied strongly for $F$, but uniqueness of the MFG equilibria does not hold for the initial measure $m_0 = \delta_0$.
	\end{proposition}

	\begin{proof}
		The proof that \eqref{LL} is satisfied strongly for $F$ is the same argument as in Proposition \ref{pr:LL not unique}.
		To prove that Nash equilibria are not unique for $m_0 = \delta_0$, set
		\begin{equation}
			\xi(t) = \frac{t^2}{2}a, \quad t \in [0,T]
		\end{equation}
		for any fixed unit vector $a \in \bb{R}^d$.
		Set $m(t) = \delta_{\xi(t)}$.
		We claim that $m(t)$ is an equilibrium.
		It suffices to show that $\xi(t)$ is an optimal trajectory for any player starting from $x = 0$ and seeking to minimize
		\begin{equation}
			\int_0^T \del{\frac{1}{2}\abs{\dot{x}(t)}^2 + \phi(x(t))\int_{\bb{R}^d} \phi \dif m(t)}\dif t - T\abs{x(T)},
		\end{equation}
		which, given $m(t) = \delta_{\xi(t)}$, can be written
		\begin{equation}
			\int_0^T \del{\frac{1}{2}\abs{\dot{x}(t)}^2 + \phi(x(t))\phi(\xi(t))}\dif t - T\abs{x(T)}.
		\end{equation}
		The Euler-Lagrange equations for this optimal control problem are
		\begin{equation} \label{eq:E-L xi}
			\ddot{x}(t) = D\phi(x(t))\phi(\xi(t)), \quad \dot{x}(T) = T \frac{x(T)}{\abs{x(T)}}.
		\end{equation}
		Then we observe that $\xi(t)$ itself satisfies \eqref{eq:E-L xi}, because
		\begin{equation}
			D\phi(\xi(t))\phi(\xi(t)) = \frac{1}{\sqrt{2\abs{\xi(t)}}}\frac{\xi(t)}{\abs{\xi(t)}}\sqrt{2\abs{\xi(t)}} = \frac{\xi(t)}{\abs{\xi(t)}} = a = \ddot{\xi}(t)
		\end{equation}
		and
		\begin{equation}
			T \frac{\xi(T)}{\abs{\xi(T)}} = Ta = \dot{\xi}(T).
		\end{equation}
		Since $a$ is arbitrary, the equilibrium is not unique.
	\end{proof}

We turn back now to the setting of Proposition \ref{pr:LL not unique}.
By a similar construction, we can give a family of examples for which both \eqref{LL} and \eqref{sigma} hold, but uniqueness of the equilibrium still does not hold.
\begin{proposition} \label{pr:sigma not unique}
	Assume the dimension $d = 1$.
	Let $\phi$ be as in Proposition \ref{pr:LL not unique}.
	In addition, assume that $\phi$ is decreasing for $x_0 < x < x_0 + z$ and increasing for $x > x_0 + z$ (where $z>0$).
	Then for the initial measure $m_0 = \delta_{x_0}$, \eqref{sigma} holds strictly, and \eqref{LL} is satisfied strongly, but uniqueness does not hold.
\end{proposition}

We emphasize that, in the statement of Proposition \ref{pr:sigma not unique}, the condition \eqref{sigma} holds only under the restriction that the initial condition must satisfy $m_0 = \delta_{x_0}$.
In the proof, we will use the following elementary lemma:
\begin{lemma} \label{lem:F+G}
	Let $F,G:I \to \bb{R}$ be continuous functions on an interval $I \subset \bb{R}$ such that $G$ is strictly increasing and let $\beta \in \argmin F\neq\emptyset$. 
	Then $\argmin (F+G) \subset I \cap \intoc{-\infty,\beta}$.
	Hence all the minimizers of $F + G$ are less than or equal to the minimizers of $F$.
\end{lemma}

\begin{proof}
	If $x > \beta$, then $F(x) \geq F(\beta)$ because $\beta$ minimizes $F$, and $G(x) > G(\beta)$ because $G$ is strictly increasing, so $F(x) + G(x) > F(\beta) + G(\beta)$.
	Thus no minimizer of $F + G$ can lie in $\intoo{\beta,\infty}$.
\end{proof}

\begin{proof}[Proof of Proposition \ref{pr:sigma not unique}]
	By Proposition \ref{pr:LL not unique}, we have only to show that \eqref{sigma} holds (here we have that $\mathcal{X}=\R$ and $\Sigma_T(m)=\sigma_T(m)$).
	For this, observe that with initial measure $m_0 = \delta_{x_0}$, we have
	\begin{equation*}
		\mathcal{E}_{T}(\sigma) = \cbr{\phi(x_0 + y) : y \in \Omega(\sigma)},
	\end{equation*}
	where
	\begin{equation}\label{def:Omega}
		\Omega(\sigma) := \argmin \cbr{\frac{y^2}{2T} + \phi(x_0 + y)\sigma : y \geq 0}.
	\end{equation}
	Note that we only need to consider $y \geq 0$ because $\phi(x_0 + y)$ is even.
	Since $\phi > 0$, we can take only $\sigma > 0$ and thus
	\begin{equation*}
		\Omega(\sigma) = \argmin \cbr{\frac{y^2}{2T\sigma} + \phi(x_0 + y) : y \geq 0}.
	\end{equation*}

	By the assumptions on $\phi$, it is obvious that $\Omega(\sigma)\neq\emptyset$. We claim that $\Omega(\sigma)$ is increasing in $\sigma$, i.e.~if $\sigma_2 > \sigma_1$ and $y_i \in \Omega(\sigma_i)$ for $i = 1,2$, then $y_2 \geq y_1$.
	Indeed, since
	\begin{equation}
		\frac{y^2}{2T\sigma_1} + \phi(x_0 + y) = \frac{y^2}{2T\sigma_2} + \phi(x_0 + y) + \del{\frac{1}{\sigma_1} - \frac{1}{\sigma_2}}\frac{y^2}{2T}
	\end{equation}
	and $y \mapsto \del{\frac{1}{\sigma_1} - \frac{1}{\sigma_2}}\frac{y^2}{2T}$ is a strictly increasing function on $\intco{0,\infty}$, the claim follows from Lemma \ref{lem:F+G}.
	
	We also claim that if $\Omega(\sigma_1) \ni z$ for some $\sigma_1 > 0$, then $\Omega(\sigma) = \{z\}$ for all $\sigma > \sigma_1$.
	Indeed, if $y \in \Omega(\sigma)$ for $\sigma \geq \sigma_1$, then since $\Omega(\sigma)$ is increasing we must have $y \geq z$.
	On the other hand, since $y \mapsto \frac{y^2}{2T\sigma} + \phi(x_0 + y)$ is strictly increasing on the interval $y \geq z$, it follows that $y \notin \Omega(\sigma)$ if $y > z$.
	The claim follows.
	
	Finally, we deduce that $\mathcal{E}_{T}(\sigma)$ is decreasing in $\sigma$.
	Indeed, let $0 < \sigma_1 < \sigma_2$ and let $y_i \in \Omega(\sigma_i)$ for $i = 1,2$.
	It follows that $y_1 \leq y_2 \leq z$.
	Since $\phi(x_0 + y)$ is decreasing on $[0,z]$, it follows that $\phi(x_0 + y_2) \leq \phi(x_0 + y_1)$, as desired.
	
	Now since $(0,+\infty)\ni\sigma\mapsto \mathcal{E}_{T}(\sigma)$ is decreasing, it follows that $(0,+\infty)\ni\sigma\mapsto  \sigma - \mathcal{E}_{T}(\sigma)$ is strictly increasing, i.e.~\eqref{sigma} holds strictly. The remaining conclusions follow directly from Proposition \ref{pr:LL not unique}.
\end{proof}

Propositions \ref{pr:LL not unique}, \ref{pr:LL really not unique}, and \ref{pr:sigma not unique} show that Lasry--Lions and even \eqref{sigma} type monotonicity conditions do not necessarily imply uniqueness of the \emph{measure}, even if the \emph{cost} is unique.
In terms of game theory, this implies that the payoffs may be uniquely determined, but the actions of the crowd are not.
One may ask whether this phenomenon is generic.
That is, when the optimal control problem does not always have a unique solution, must there always be more than one equilibrium?
We conclude this section by showing an example where this is not the case: although the optimal control problem for individuals may sometimes have more than one solution, nevertheless there is only one possible equilibrium in the game.

\subsection{Non-uniqueness of feedback strategies in general does not imply non-uniqueness of Nash equilibria}

We again work in dimension $d = 1$.
For the data, let us take
\begin{equation} \label{eq:special case1}
	G(x,m) = \phi(x)\sigma_T(m), \quad \sigma_T(m) = \int_{\R} \psi \dif m, \quad L(x,v) = \frac{1}{2}v^2,
\end{equation}
where $\psi$ is any function that is strictly increasing and continuous, and where $\phi$ is given by
\begin{equation} \label{eq:special case2}
	\phi(x) = \begin{cases}
		x, &\text{if}~x \leq 0,\\
		2x, &\text{if}~0 < x < 1,\\
		x+1, &\text{if}~x \geq 1.
	\end{cases}
\end{equation}
\begin{remark}
		Although $\phi$ and $\psi$ are both increasing functions, nevertheless it does not generally follow that $G$ is LL monotone.
Indeed, since $\phi(x) = x$ for $x \leq 0$, one may take any increasing function $\psi$ that is nonlinear on $\intoc{-\infty,0}$, then apply the argument found in the proof of Proposition \ref{pr:sigma not LLDX} below to conclude that $G$ is not LL monotone.
\end{remark}

\begin{theorem} \label{thm:special case}
	Let $m_0 \in \sr{P}_2(\bb{R})$ be any initial measure, and let the data be given by \eqref{eq:special case1} and \eqref{eq:special case2}.
	Then there exists a unique equilibrium measure $m$, i.e.~there is a unique fixed point $m \in \s{E}_T(m)$.
\end{theorem}

\begin{remark}
	It seems to us that the conclusion of Theorem \ref{thm:special case} holds for a much more general class of increasing functions $\phi$.
	We have chosen to keep the structure simple so as not to obscure the main idea of the proof.
\end{remark}

\begin{proof}[Proof of Theorem \ref{thm:special case}]
	Define
\begin{equation}
	\begin{split}
		\Phi(x,y,\sigma) &= \frac{(x-y)^2}{2T} + \phi(y)\sigma, \quad Y(x,\sigma) = \argmin \Phi(x,\cdot,\sigma),\\
		y_*(x,\sigma) &= \min Y(x,\sigma), \quad y^*(x,\sigma) = \max Y(x,\sigma).
	\end{split}
\end{equation}
Using Lemma \ref{lem:F+G}, the strict convexity of the square, and the fact that $\phi$ is strictly increasing,
we deduce that
\begin{equation}\label{y_inc_sigma}
	y^*(x,\sigma_2) \leq y_*(x,\sigma_1) \quad \forall \sigma_2 \geq \sigma_1, \quad y^*(x_2,\sigma) \geq y_*(x_1,\sigma) \quad \forall x_2 \geq x_1.
\end{equation}

We claim that there is at most one $x$ such that $\Phi(x,\cdot,\sigma)$ has two minimizers.
To see this, we first compute
\begin{equation}
	\partial_y \Phi(x,y,\sigma) = \frac{1}{T}\begin{cases}
		y - (x-2\sigma T) & \text{if}~y \in (0,1),\\
		y - (x-\sigma T) & \text{if}~y \notin [0,1].
	\end{cases}
\end{equation}
We divide into two cases.
By examining the sign of the derivative, we can deduce intervals of increase and decrease to identify candidates for minimizers.
\begin{enumerate}
	\item First assume $\sigma T < 1$.
	If $x \leq 2\sigma T$, then $\Phi(x,\cdot,\sigma)$ has a unique minimizer, $\min\{x-\sigma T,0\}$, while if $2\sigma T < x \leq 1 + \sigma T$, then the unique minimizer is $x-2\sigma T$.
	If $x \geq 1 + 2\sigma T$, then the unique minimizer is $x - \sigma T$.
	
	In the remaining case where $1 + \sigma T < x < 1 + 2\sigma T$, both $x - 2\sigma T$ and $x - \sigma T$ are candidates for minimizer.
	We compute
	\begin{equation} \label{eq:min Phi}
		\begin{split}
			\Phi(x,x-2\sigma T,\sigma) &= 2\sigma x - 2T\sigma^2,\\
			\Phi(x,x-\sigma T,\sigma) &= \sigma (x+1) - \frac{T\sigma^2}{2},\\
		\end{split}
	\end{equation}
	and these are equal if and only if $x = 1 + \frac{3\sigma T}{2}$, in which case it follows that $1 \pm \frac{\sigma T}{2}$ are both minimizers.
	\item Now assume, to the contrary, that $\sigma T \geq 1$.
	If $x \leq 2\sigma T$, then the unique minimizer is $\min\{x-\sigma T,0\}$.
	If $x \geq 1 + 2\sigma T$, then the unique minimizer is $x-\sigma T$.
	
	In the remaining case where $2\sigma T \leq x < 1 + 2\sigma T$, both $x - 2\sigma T$ and $x - \sigma T$ are candidates for minimizer.
	We get the same values as in \eqref{eq:min Phi}.
\end{enumerate}
We conclude that there are exactly two distinct minimizers, $1 \pm \frac{\sigma T}{2}$, if and only if $x = 1 + \frac{3\sigma T}{2}$; otherwise, there is only one minimizer.
In other words, for $x \neq 1 + \frac{3\sigma T}{2}$, $y_*(x,\sigma) = y^*(x,\sigma)$, but on the other hand
\begin{equation*}
	Y\del{1 + \frac{3\sigma T}{2},\sigma} = \cbr{1 + \frac{\sigma T}{2},1 - \frac{\sigma T}{2}}.
\end{equation*}

In what follows it will be useful to note that in all cases,
\begin{equation} \label{eq:bounds on y*}
	\abs{y_*(x,\sigma)}, \abs{y^*(x,\sigma)} \leq \abs{x} + 2T\abs{\sigma}.
\end{equation}
	
Let $x_\sigma:=1 + \frac{3\sigma T}{2}.$ If $m_0(\{x_\sigma\})=0$ or $\sigma=0$ (in which case one has that $y_*(x_0,0)=y^*(x_0,0)$) all points $x$ from the support of $m_0$ have a unique destination $y_*(x,\sigma)=y^*(x,\sigma)$. In such case the target measure has the form $y_*(\cdot,\sigma)_\sharp m_0$ and so
$$
\mathcal{E}_{T}(\sigma) = \int_{\R} \psi\del{y_*(x,\sigma)}\dif m_0(x).
$$
As $\psi$ is strictly increasing and $\sigma\mapsto y_*(x,\sigma)$ is decreasing for any $x$ (cf. \eqref{y_inc_sigma}), we have that $\mathcal{E}_{T}$ is decreasing.

If $m_0(\{x_\sigma\}) >0$ and $\sigma\neq 0$, then from the point $x_\sigma$, $cm_0(\{x_\sigma\})$ amount of mass can travel to $1 + \frac{\sigma T}{2}$ while all the remaining $(1-c)m_0(\{x_\sigma\})$ amount of mass travels to $1 - \frac{\sigma T}{2}$, for an arbitrary $c\in[0,1].$ All the remaining points $x$ from the support of $m_0$ have a unique destination $y_*(x,\sigma)$. Therefore the target measure has the form 
\begin{equation} \label{eq:equil cond special case}
	m = y_*(\cdot,\sigma)_\sharp \left(m_0\mres(\R\setminus\{x_\sigma\})\right)+m_0(\{x_\sigma\})\del{c\delta_{1 + \frac{\sigma T}{2}}+(1-c)\delta_{1 - \frac{\sigma T}{2}}},
\end{equation}
where $m_0\mres{E}$ denotes the measure $m_0$ restricted to a set $E$ and $\delta_z$ is a Dirac mass concentrated at $z$.
Therefore, as $\psi$ is strictly increasing, $\mathcal{E}_{T}(\sigma)$ becomes set valued, i.e.
\begin{align*}
\mathcal{E}_{T}(\sigma) &=  \cbr{\int_{\R\setminus\{x_\sigma\}} \psi\del{y_*(x,\sigma)}\dif m_0(x) + m_0(\{x_\sigma\})\left[c\psi\left(1 + \frac{\sigma T}{2}\right)+(1-c)\psi\left(1 - \frac{\sigma T}{2}\right)\right] : c \in [0,1]}\\
& = \cbr{\int_{\R} \del{c\psi\del{y_*(x,\sigma)} + (1-c)\psi\del{y^*(x,\sigma)}}\dif m_0(x) : c \in [0,1]},
\end{align*}
which can be written more simply as the interval
\begin{equation*}
	\mathcal{E}_{T}(\sigma) = \intcc{\int_{\R} \psi\del{y_*(x,\sigma)}\dif m_0(x), \int_{\R} \psi\del{y^*(x,\sigma)}\dif m_0(x)},
\end{equation*}
where we have used the fact that $\psi\del{y_*(x,\sigma)} = \psi\del{y^*(x,\sigma)}$ for $x \neq x_\sigma$.
Again, by \eqref{y_inc_sigma}, we see that $\s{E}_T(\sigma)$ is decreasing.
We claim also that it is maximal, i.e.~if $(\sigma,\tau)$ is any pair satisfying
\begin{equation} \label{eq:monotonicity}
	(\tilde\tau -  \tau)(\tilde\sigma -  \sigma) \leq 0 \quad \forall \tilde \sigma \in \bb{R}, \tilde \tau \in \s{E}_T(\tilde \sigma),
\end{equation}
then it follows that $\tau \in \s{E}_T(\sigma)$.
Thus suppose \eqref{eq:monotonicity} holds.
Let $\sigma_n$ be a sequence that increases to $\sigma$ as $n \to \infty$.
Since $\tau_n = \int_{\R} \psi\del{y_*(x,\sigma_n)}\dif m_0(x) \in \s{E}_T(\sigma_n)$, \eqref{eq:monotonicity} implies
\begin{equation*}
	(\tau_n - \tau)(\sigma_n - \sigma) \leq 0 \ \forall n \quad \Rightarrow \tau \leq \tau_n \ \forall n.
\end{equation*}
Let us show that $\ds\liminf_{n\to+\infty} \tau_n \leq \int_{\R} \psi\del{y^*(x,\sigma)}\dif m_0(x)$.
By \eqref{eq:bounds on y*} and the dominated convergence theorem, it is enough to show that
\begin{equation}
	\liminf_{n\to+\infty} \psi\del{y_*(x,\sigma_n)} \leq \psi\del{y^*(x,\sigma)}
\end{equation}
for all $x$, and since $\psi$ is strictly increasing this is equivalent to
\begin{equation} \label{eq:liminf}
	\liminf_{n\to+\infty} y_*(x,\sigma_n) \leq y^*(x,\sigma).
\end{equation}
(In fact, we get equality, the opposite inequality being trivial.)
Let $y_n = y_*(x,\sigma_n)$ for some fixed $x$; since $\sigma_n \to \sigma$, $y_n$ is bounded by \eqref{eq:bounds on y*}.
Without relabeling, we pass to a subsequence $y_n$ that converges to some $y_\infty$.
Observe that, since $y_n \in Y(x,\sigma_n)$ (i.e.~$y_n$ is a minimizer of $\Phi(x,\cdot,\sigma_n)$),
\begin{equation}
	\Phi(x,y_n,\sigma) = \Phi(x,y_n,\sigma_n) - \Phi(x,y_n,\sigma_n) + \Phi(x,y_n,\sigma)
	\leq \Phi(x,y,\sigma_n) - \Phi(x,y_n,\sigma_n) + \Phi(x,y_n,\sigma) \quad \forall y,
\end{equation}

hence, letting $n \to \infty$ and using the continuity of $\Phi$, we get
\begin{equation}
	\Phi(x,y_\infty,\sigma) \leq \Phi(x,y,\sigma) \quad \forall y.
\end{equation}
It follows that $y_\infty \in Y(x,\sigma)$, so $y_\infty \leq y^*(x,\sigma)$.
The claim \eqref{eq:liminf} follows, from which we deduce, in turn, that 
\begin{equation}
	\tau \leq \liminf \tau_n \leq \int_{\R} \psi\del{y^*(x,\sigma)}\dif m_0(x).
\end{equation}
We next need to prove $\tau \geq \int_{\R} \psi\del{y_*(x,\sigma)}\dif m_0(x)$.
This is entirely analogous.
We take $\sigma_n$ decreasing to $\sigma$, then let $\tau_n = \int_{\R} \psi\del{y^*(x,\sigma_n)}\dif m_0(x)$ and are able to prove, by the mirror image of the same arguments, that
\begin{equation}
	\tau \geq \limsup \tau_n \geq \int_{\R} \psi\del{y_*(x,\sigma)}\dif m_0(x),
\end{equation}
as desired.
It follows that $\tau \in \s{E}_T(\sigma)$, which is what we wanted to show.
Since $\s{E}_T$ is a maximal decreasing set-valued function, it has a unique fixed point $\sigma \in \s{E}_T(\sigma)$.
From now on, we take $\sigma$ to be this fixed point.

If $m$ is any equilibrium measure,  i.e.~if $m \in E_T(m)$, then it must have the form \eqref{eq:equil cond special case} for some $c \in [0,1]$.
We now show that the equilibrium measure is unique.
In particular  the constant $c \in [0,1]$ in \eqref{eq:equil cond special case} is uniquely determined by the equilibrium condition.
\begin{enumerate}
	\item First consider the case $\sigma = 0$.
	In this case, there is in fact no mass splitting, and the formula \eqref{eq:equil cond special case} is independent of $c$ and completely determines $m$.
	\item We now suppose $\sigma \neq 0$, and we show that the constant $c$ is uniquely determined by the equilibrium condition.
	We must have
	\begin{equation*}
		\sigma = \sigma_T(m) = \int_{\R} \psi(x)\dif m(x),
	\end{equation*}
	which by \eqref{eq:equil cond special case} implies
	\begin{equation} \label{eq:c determined}
		\sigma = \int_{\{x_\sigma\}^c} \psi(y_*(x,\sigma))\dif m_0(x) + m_0\del{\cbr{x_\sigma}}\del{c\psi\del{1 + \frac{\sigma T}{2}} + (1-c)\psi\del{1 - \frac{\sigma T}{2}}}.
	\end{equation}
	Since $\psi$ is strictly increasing and $\sigma \neq 0$, it follows that $\psi\del{1 + \frac{\sigma T}{2}}$ and $\psi\del{1 - \frac{\sigma T}{2}}$ are distinct.
	Therefore \eqref{eq:c determined} uniquely defines $c$, which in turn uniquely defines the equilibrium measure $m$.
\end{enumerate}
\end{proof}

\begin{remark}
The proof of Theorem \ref{thm:special case} relies on a special construction of the data to make minimizers of the optimal control problem explicit.
More abstractly, we have used the following scheme:
\begin{enumerate}
	\item prove that there is a unique equilibrium point $\sigma \in \mathcal{E}_{T}(\sigma)$;
	\item prove that, for this $\sigma$, the equation $\sigma_T(m) = \sigma$ has only one solution.
\end{enumerate}
To prove the second point, it was useful to assume a sort of monotonicity property; in this case, $\sigma_T(m) = \int_{\R} \psi \dif m$ with $\psi$ monotone.
Note, however, that this is generally insufficient; we also needed the fact that for every $x$ but one, there was a unique optimal trajectory starting from $x$.
Thus one can see that in general, the issue of uniqueness can be rather complex, but not insurmountable.
\end{remark}

\section{No implications between conditions} \label{sec:no implications}

The purpose of this section is to highlight simple examples in which each of the four conditions given in Section \ref{sec:4 conditions} might hold, and to show that none of them necessarily implies any of the others.
For simplicity, let us assume $L(x,v) = \frac{1}{2}\abs{v}^2$, so that optimal trajectories starting from $x$ are straight lines ending up at a point $y$ satisfying
\begin{equation} \label{eq:1st order cond}
	y + T D_x G(y,m) = x,
\end{equation}
where $m$ is the final measure.
If we consider $(I + T D_x G(\cdot,m))^{-1}$ as a multi-valued function, where $I$ is the identity map, then \eqref{eq:E1hat} and \eqref{eq:E2} can be rewritten now as follows.
For a given initial random variable $X_0$ whose law is $m_0$, we have
\begin{equation}
	\tilde E_T(X) = \cbr{Y \in \bb{H} : Y \in \del{I + T D_x G(\cdot,\s{L}_X)}^{-1}(X_0) \text{ a.s.}}
\end{equation}
Similarly, if $G(x,m) = g(x,\sigma_T(m))$, recalling the definition \eqref{def:Omega}, we can write
\begin{equation}
	\Omega(\sigma) = \cbr{\sigma_T(\s{L}_Y) : Y \in \bb{H}, Y \in \del{I + T D_x g(\cdot,\sigma)}^{-1}(X_0) \text{ a.s.}}.
\end{equation}
In each of the examples below, we construct a function $G(x,m)$ such that at least one of the conditions given in Section \ref{sec:4 conditions} holds, but at least one another does not.
Through these formulas we can understand \eqref{X} and \eqref{sigma} as truly being properties of $G$ itself, as is evident in the case of conditions \eqref{LL} and \eqref{D}.
Our task now is to show that these properties have no necessary implications between them. First of all, we can formulate the following well-known theorem.

\begin{theorem}
The condition \eqref{LL} in general does not imply \eqref{D}, nor does \eqref{D} in general imply \eqref{LL}.
\end{theorem}
\begin{proof}
The proof of this result is well-known now in the literature, we refer to \cite{gangbo2020global,gangbo2021mean,meszaros2021mean} for the details.
\end{proof}

\subsection{\eqref{LL} does not imply \eqref{X} or \eqref{sigma}}

Here we show that \eqref{LL} can hold even when neither \eqref{X} nor \eqref{sigma} holds.
Actually, we show that none of the conditions \eqref{X}, -\eqref{X}, \eqref{sigma}, and -\eqref{sigma} hold.
In other words, changing monotonicity to ``anti-monotonicity'' does not affect the overall result, which is that Lasry-Lions monotonicity does not imply either of these other kinds of monotonicity.
We will abide by this same pattern in the following subsection: to say that a condition does not imply \eqref{X} (resp.~\eqref{sigma}) is also to say that it does not imply -\eqref{X} (resp.~-\eqref{sigma}).

\begin{proposition}
Let $d = 1$, let $f:\R\to\R$ be given, with at most cubic growth at $\pm\infty$. Let $g:\R\times\R\to\R$ be given by $g(x,\sigma) = f(x)\sigma$, $\sigma:\sP_3(\R)\to\R$ is defined as $\sigma_T(m) = \int_{\R} f \dif m$, and let $G:\R\times\sP_3(\R)\to\R$ be given $G(x,m) = g(x,\sigma_T(m))$.
\begin{enumerate}
\item Then $G$ is always Lasry--Lions monotone, regardless of the choice of $f$.
\item Let $f(x) = \frac{1}{3}x^3$. Then $G$ satisfies neither $\pm$\eqref{X} nor $\pm$\eqref{sigma}.
\end{enumerate}
\end{proposition}

\begin{proof}
(1) this is immediate, using the definition. Indeed,
\begin{equation*}
	\int_{\R} \del{G(x,m_1) - G(x,m_2)}\dif\, (m_1 - m_2)(x) = \del{\int_{\R} f \dif\, (m_1 - m_2)}^2 \geq 0.
\end{equation*}

(2) We note that the first order condition \eqref{eq:1st order cond} becomes

\begin{equation}
	y + T\sigma y^2 = x,
\end{equation}
whose solution set is
\begin{equation}
	y = \frac{-1 \pm \sqrt{1 + 4T\sigma x}}{2T\sigma}
\end{equation}
whenever $1 + 4T\sigma x \geq 0$.
Let us fix $m_0$ to be the Dirac mass $\delta_1$, and let $X_0 = 1$ a.s.
Then $\tilde E_T(X)$ consists of the set of all random variables $Y$ such that for a.s.~$\omega$, we have
\begin{equation}
	Y(\omega) \in \cbr{\frac{-1 + \sqrt{1 + 4T\bb{E}\sbr{\frac{1}{3}X^3}X_0(\omega)}}{2T\bb{E}\sbr{\frac{1}{3}X^3}}, \frac{-1 - \sqrt{1 + 4T\bb{E}\sbr{\frac{1}{3}X^3}X(\omega)}}{2T\bb{E}\sbr{\frac{1}{3}X^3}}}.
\end{equation}
Now for each $\sigma > 0$ define $X_\sigma$ to be the constant random variable $\sigma$, and set
\begin{equation}
	Y_\sigma^\pm = \frac{-1 \pm \sqrt{1 + \frac{4}{3}T{\sigma}^{3}}}{\frac{2}{3}T{\sigma}^3}.
\end{equation}
Note that $Y_\sigma^\pm \in \tilde E_T(X_\sigma)$.
Observe that
\begin{equation*}
	\lim_{\tilde{\sigma} \to \sigma} Y_{\tilde{\sigma}}^{+} - Y_{\sigma}^- = \frac{3\sqrt{1 + \frac{4}{3}T{\sigma}^{3}}}{T{\sigma}^3} > 0.
\end{equation*}
Therefore, for $\tilde{\sigma} > \sigma$ small enough, we have
\begin{equation*}
	\bb{E}\sbr{(Y_{\tilde{\sigma}}^{+} - Y_{\sigma}^-)(X_{\tilde{\sigma}} - X_\sigma)}
	= (Y_{\tilde{\sigma}} - Y_{\sigma}^-)(\tilde{\sigma} - \sigma)
	> (\tilde{\sigma} - \sigma)^2 = \bb{E}\abs{X_{\tilde{\sigma}} - X_\sigma}^2.
\end{equation*}
Therefore \eqref{X} does not hold.
Likewise -\eqref{X} does not hold, because for $\tilde\sigma>\sigma$ small enough
\begin{equation*}
	\bb{E}\sbr{(Y_{\tilde{\sigma}}^- - Y_{\sigma}^{+})(X_{\tilde{\sigma}} - X_\sigma)}
	= (Y_{\tilde{\sigma}}^- - Y_{\sigma}^{+})(\tilde{\sigma} - \sigma)
	< 0.
\end{equation*}

We argue in similar fashion to see that neither \eqref{sigma} nor -\eqref{sigma} holds.
Indeed, if we let
\begin{equation*}
	y_\sigma^\pm = \frac{-1 \pm \sqrt{1 + 4T\sigma}}{2T\sigma}
\end{equation*}
and consider the Dirac masses $\delta_{y_\sigma^\pm}$, we see that each is simply a push-forward of the initial mass $\delta_1$ onto an optimal point.
Thus,
\begin{equation*}
	\sigma_T(\delta_{y_\sigma^\pm}) \in \Omega(\sigma),
\end{equation*}
We compute
\begin{equation*}
	\sigma_T(\delta_{y_\sigma^\pm}) = \frac{1}{3}\int_{\R} x^3 \dif \delta_{y_\sigma^\pm}(x) = \frac{1}{3}(y_\sigma^\pm)^3 = \frac{\pm (1+T\sigma)\sqrt{1+4T\sigma} - 3T\sigma - 1}{6(T\sigma)^3} =: \psi^\pm(\sigma).
\end{equation*}
Again we have
\begin{equation*}
	\lim_{\tilde{\sigma} \to \sigma} \psi^{+}(\tilde{\sigma}) - \psi^-(\sigma) = \frac{(1+T\sigma)\sqrt{1+4T\sigma}}{3(T\sigma)^3} > 0.
\end{equation*}
Hence for $\tilde{\sigma} > \sigma$ sufficiently small, we have
\begin{equation*}
	\psi^{+}(\tilde{\sigma}) - \psi^-(\sigma) > \tilde{\sigma} - \sigma
	\quad \text{and} \quad
	\psi^-(\tilde{\sigma}) - \psi^{+}(\sigma) < 0.
\end{equation*}
It follows that neither \eqref{sigma} nor -\eqref{sigma} holds.
\end{proof}

\subsection{\eqref{D} does not imply \eqref{X} or \eqref{sigma}}
In this subsection, we will consider cases where $G$ is displacement monotone, i.e. it satisfied \eqref{D}. It is well-known (see \cite[Lemma 2.6]{gangbo2021mean} and \cite[Lemma 2.3]{meszaros2021mean}) that this property implies that $x\mapsto G(x,m)$ is convex for all $m\in\sP_2(\R^d)$, and thus \eqref{eq:1st order cond} can be uniquely solved by
\begin{equation*}
	y = \del{I + T D_x G(\cdot,m)}^{-1}(x).
\end{equation*}
It follows that $\tilde E_T$ is a single-valued function; for a given initial random variable $X_0$ whose law is $m_0$ it is given by
\begin{equation} \label{eq:hat E1 1}
	\tilde E_T(X) = \del{I + T D_x G(\cdot,\s{L}_X)}^{-1}(X_0).
\end{equation}
Assuming the structure $G(x,m) = g(x,\sigma_T(m))$, we likewise have that $\Omega(\sigma)$ is single-valued with
\begin{equation} \label{eq:E2 1}
	\Omega(\sigma) = \sigma_T\del{\del{I + T D_x g(\cdot,\sigma)}^{-1}_\sharp m_0}.
\end{equation}
Equations \eqref{eq:hat E1 1} and \eqref{eq:E2 1} provide explicit formulas for $\tilde E_T$ and $\Omega$ in terms of $g$ (or $G$).

Let us take $d = k = 1$.
Recall \cite{gangbo2021mean,meszaros2021mean} that $G$ is displacement monotone if and only if
\begin{equation} \label{eq:displacement mon}
	\int_{\R} \partial_{xx}^2 G(x,m)v(x)^2 \dif m(x) + \iint_{\R\times\R} \partial_{xm}^2 G(x,m,y)v(x)v(y)\dif m(x)\dif m(y) \geq 0, 
\end{equation}
$\forall m \in \sr{P}_2(\R), \ \forall v \in L^2_m(\R).$
We suppose $G$ has the structure $G(x,m) = g(x,\sigma_T(m))$ where $\sigma_T:\sr{P}_2(\R) \to \bb{R}$.
Then \eqref{eq:displacement mon} becomes
\begin{equation} \label{eq:displacement mon1}
	\int_{\R} \partial_{xx}^2 g(x,\sigma_T(m))v(x)^2 \dif m(x) + \int_\R \partial_{x\sigma}^2 g(x,\sigma_T(m))v(x) \dif m(x)\int_\R D_m \sigma_T(m,y)v(y)\dif m(y) \geq 0,
\end{equation}
$ \quad \forall m \in \sr{P}_2(\R), \ \forall v \in L^2_m(\R).$ In the view of this, we can state the following result.

\begin{proposition}
Let $g:\R\times\R\to\R$ given by $g(x,\sigma) = \frac{1}{2} x^2 \phi(\sigma)$ and let $\sigma_T:\sP_2(\R)\to\R$ defined as $\sigma_T(m) = \frac{1}{2}\int_{\R} x^2 \dif m(x)$, where $\phi:\bb{R} \to \bb{R}$ is a smooth, non-negative function that satisfies
\begin{equation} \label{eq:diff ineq}
	\phi(\sigma) \geq -2\sigma \phi'(\sigma).
\end{equation}
Then, $G:\R\times\sP_2(\R)\to\R$ given by $G(x,m)=g(x,\sigma_T(m))$ satisfies \eqref{D}; however, in general it does not satisfy $\pm$\eqref{X} or $\pm$\eqref{sigma}.
\end{proposition} 

\begin{proof}
With our specific structural condition, the inequality \eqref{eq:displacement mon1} that we need to check, becomes
\begin{equation} \label{eq:displacement mon2}
	\phi\del{\frac{1}{2}\int_{\R} x^2 \dif m(x)}\int_\R v(x)^2 \dif m(x) + \phi'\del{\frac{1}{2}\int_\R x^2 \dif m(x)}\del{\int_\R xv(x) \dif m(x)}^2 \geq 0, \ \forall m \in \sr{P}_2(\R), \ \forall v \in L^2_m(\R).
\end{equation}

Now we claim that \eqref{eq:diff ineq} implies \eqref{eq:displacement mon2}. To see this, note that if $\phi'\del{\frac{1}{2}\int_{\R} x^2 \dif m(x)} \geq 0$ then there is nothing to show because both terms are non-negative.
If $\phi'\del{\frac{1}{2}\int_{\R} x^2 \dif m(x)} < 0$, then by \eqref{eq:diff ineq} and the Cauchy--Schwarz inequality, we obtain
\begin{equation*}
	\begin{split}
		\phi\del{\frac{1}{2}\int_{\R} x^2 \dif m(x)}\int_{\R} v(x)^2 \dif m(x)
		&\geq -\phi'\del{\frac{1}{2}\int_{\R} x^2 \dif m(x)}\int_{\R} x^2 \dif m(x)\int_{\R} v(x)^2 \dif m(x)\\
		&\geq -\phi'\del{\frac{1}{2}\int_{\R} x^2 \dif m(x)}\del{\int_{\R} xv(x) \dif m(x)}^2,
	\end{split}
\end{equation*}
which implies \eqref{eq:displacement mon2}.

Let us now show that, even if \eqref{eq:diff ineq} and therefore \eqref{D} hold, \eqref{X} may not hold.
Now \eqref{eq:hat E1 1} becomes
\begin{equation*}
	\tilde E_T(X) = \frac{X_0}{1 + T\phi\del{\frac{1}{2}\bb{E}[X^2]}}.
\end{equation*}
We differentiate to get
\begin{equation*}
	D\tilde E_T(X)Y = -\frac{X_0}{\del{1 + T\phi\del{\frac{1}{2}\bb{E}[X^2]}}^2}T\phi'\del{\frac{1}{2}\bb{E}[X^2]} \bb{E}[XY].
\end{equation*}
Let $\phi$ be any smooth non-negative function on $\intco{0,\infty}$ that satisfies \eqref{eq:diff ineq} and equals $\sigma^{-1/2}$ on $\intco{1,\infty}$.
Then as long as $\frac{1}{2}\bb{E}[X^2] \geq 1$ we have
\begin{equation}
	\bb{E}\sbr{\del{D\tilde E_T(X)Y}Y} = \frac{\sqrt{2}T\bb{E}[X_0 Y]\bb{E}[XY]}{\del{\bb{E}[X^2]^{1/2} + \sqrt{2}T}^2 \bb{E}[X^2]^{1/2}}.
\end{equation}

Choose $X_0$ (i.e.~choose $m_0$ with $\mathcal{L}_{X_0}=m_0$) and $T$ such that
\begin{equation*}
	(\bb{E}[X_0^2])^{\frac12} > \frac{\sqrt{2}(1+T)^2}{T}.
\end{equation*}
Let $X = \alpha \frac{X_0}{\bb{E}[X_0^2]^{1/2}}$ for some parameter $\alpha \geq \sqrt{2}$, and let $Y = \frac{X_0}{\bb{E}[X_0^2]^{1/2}}$.
Then we have
\begin{equation*}
	\bb{E}\sbr{\del{D\tilde E_T(X)Y}Y} = \frac{\sqrt{2}T (\bb{E}[X_0^2])^{\frac12}}{\del{\alpha + \sqrt{2}T}^2}.
\end{equation*}
Since
\begin{equation*}
	\lim_{\alpha \to \sqrt{2}} \frac{\sqrt{2}T (\bb{E}[X_0^2])^{\frac12}}{\del{\alpha + \sqrt{2}T}^2} = \frac{T(\bb{E}[X_0^2])^{\frac12}}{\sqrt{2}(1+T)^2} > 1,
\end{equation*}
we see that for some $X$ and $Y$, we have $\bb{E}\sbr{\del{D\tilde E_T(X)Y}Y} > 1 = \bb{E}[Y^2]$.
Thus \eqref{X} does not hold.
On the other hand, by taking $\alpha$ large we can also have $\bb{E}\sbr{\del{D\tilde E_T(X)Y}Y} < \bb{E}[Y^2]$, and thus -\eqref{X} does not hold, either.

Similarly, let us now show that, even if \eqref{eq:diff ineq} and therefore \eqref{D} hold, \eqref{sigma} may not hold.
Now \eqref{eq:E2 1} becomes
\begin{equation*}
	\Omega(\sigma) = \frac{1}{2}\int_{\R} \frac{x^2}{\del{1 + T\phi(\sigma)}^2}\dif m_0(x).
\end{equation*}
Differentiating, we get
\begin{equation*}
	\Omega'(\sigma) = -\int_{\R} \frac{x^2T\phi'(\sigma)}{\del{1 + T\phi(\sigma)}^3}\dif m_0(x).
\end{equation*}
Let us choose $m_0$ and $T$ such that
\begin{equation}
	\frac{T}{2(1+T)^3} \int_{\R} x^2 \dif m_0(x) > 1.
\end{equation}
Recall that $\phi$ is any smooth non-negative function on $\intco{0,\infty}$ that satisfies \eqref{eq:diff ineq} and equals $\sigma^{-1/2}$ on $\intco{1,\infty}$, which can be made even smaller if necessary.
We have
\begin{equation}
	\Omega'(\sigma) = \frac{T}{2(\sigma^{1/2} + T)^3}\int x^2 \dif m_0(x) \quad \forall \sigma \geq 1.
\end{equation}
Since
\begin{equation}
	\frac{T}{2(\sigma^{1/2} + T)^3}\int x^2 \dif m_0(x) \to \frac{T}{2(1+T)^3} \int x^2 \dif m_0(x) > 1 \quad \text{as} \quad \sigma \to 1+,
\end{equation}
we have $\Omega'(\sigma) > 1$ for $\sigma$ close to $1$.
It follows that \eqref{sigma} does not hold.
On the other hand, we will also have $\lim_{\sigma \to \infty} \Omega'(\sigma) = 0$, so $\Omega'(\sigma) < 1$ for all $\sigma$ large enough.
Thus -\eqref{sigma} does not hold either.
\end{proof}

\subsection{\eqref{X} does not imply \eqref{LL} or \eqref{D}}

Assume $G(x,m) = f(x)h(m)$ for some $f:\bb{R}^d \to \bb{R}$ and $h:\sr{P}_2(\bb{R}^d) \to \bb{R}$.
We will assume that $f$ and $h$ are smooth, and that $f$ is convex.
Equation \eqref{eq:hat E1 1} becomes
\begin{equation}
	\tilde E_T(X) = \del{I + Th(\s{L}_X)D f(\cdot)}^{-1}(X_0).
\end{equation}
By the implicit function theorem, one deduces that $\tilde E_T:\bb{H} \to \bb{H}$ is differentiable, and by implicit differentiation, one derives that
\begin{equation} \label{eq:DhatE1}
	D\tilde E_T(X)Y = -T\del{I + Th(\s{L}_X)D^2 f\del{\tilde E_T(X)}}^{-1}Df\del{\tilde E_T(X)} \bb{E}\sbr{D_m h(\s{L}_X,X)Y} \quad \forall X,Y \in \bb{H}.
\end{equation}

\begin{proposition} \label{prop:fhX}
	Suppose that $G(x,m) = f(x)h(m)$ for some $f:\bb{R}^d \to \bb{R}$ and $h:\sr{P}_2(\bb{R}^d) \to \bb{R}$.
	Assume that $h$ is continuously differentiable, and that for some $a,b > 0$ the following estimates hold:
	\begin{equation} \label{eq:h assumption}
		h(m) \geq a, \quad \abs{D_m h(m,x)} \leq b, \quad \forall m \in \sr{P}_2(\R^d), \forall x \in \bb{R}^d.
	\end{equation}
	Assume that $f:\R^d\to\R$ is $\s{C}^2$ smooth and convex, and that the following estimate holds:
	\begin{equation} \label{eq:f assumption}
		D^2 f(x) \geq \frac{b}{a}\abs{Df(x)} \quad \forall x \in \bb{R}^d.
	\end{equation}
	Then \eqref{X} is satisfied.
\end{proposition}

\begin{proof}
	Using \eqref{eq:DhatE1} and the estimates on $h$ and $f$, we get
	\begin{equation}
		\begin{split}
			\bb{E}\sbr{(D\tilde E_T(X)Y) \cdot Y} &= -T\bb{E}\sbr{\del{\del{I + Th(\s{L}_X)D^2 f\del{\tilde E_T(X)}}^{-1}Df\del{\tilde E_T(X)}} \cdot Y} \bb{E}\sbr{D_m h(\s{L}_X,X)Y}\\
			&\leq \bb{E}\sbr{\del{\del{I + Ta \frac{b}{a}\abs{D f\del{\tilde E_T(X)}}}^{-1}T \abs{Df\del{\tilde E_T(X)}}} \abs{Y}} \bb{E}\sbr{b\abs{Y}}\\
			&\leq \bb{E}\sbr{\abs{Y}}^2 \leq \bb{E}\sbr{\abs{Y}^2} \quad \forall X,Y \in \bb{H},
		\end{split}
	\end{equation}
	from which we deduce \eqref{X}.
\end{proof}

\begin{remark}
Let us now remark on the applicability of Proposition \ref{prop:fhX}. The assumptions of this proposition, namely \eqref{eq:f assumption}, imply that $Df$ and $D^2 f$ grow exponentially and are therefore unbounded.

If one wishes to impose other type of potential assumptions (that are more natural in the literature on mean field games), such as $D_xG$ is Lipschitz continuous or has a linear growth at infinity, it is still possible to satisfy a conditional version of \eqref{X}.
	Namely, we fix a bounded set $B \subset \bb{R}^d$ and consider only those initial measures $m_0$ supported in $B$, hence only initial random variables $X_0$ with values in $B$.
	Then there exists another bounded set $B' \subset \bb{R}^d$ such that $\tilde E_T(X)$ necessarily takes values only in $B'$.
	In this case, the estimate \eqref{eq:f assumption} would only be required to hold for $x \in B'$, and it would still follow that \eqref{X} holds.
\end{remark}

We now construct an example satisfying \eqref{X} but neither \eqref{LL} nor \eqref{D}.
\begin{proposition}
Let $d = 1$, $f(x) = e^{-x}$, and $h(m) = 2 + \int_{\R} \sin x \dif m(x)$, and assume that $G:\R\times\sP_2(\R)\to\R$ is given by $G(x,m) = f(x)h(m)$. Then $G$ satisfies \eqref{X}, but it does not satisfy \eqref{D} and \eqref{LL}.
\end{proposition}

\begin{proof}
Notice that the hypotheses of Proposition \ref{prop:fhX} are satisfied, with $a = b = 1$, and thus \eqref{X} holds.

Let us show that \eqref{D} does not hold.
Recall from \cite{gangbo2021mean,meszaros2021mean} that \eqref{D} is equivalent to
\begin{equation*}
	\bb{E}\sbr{\del{D_{xx}^2 G(X,\s{L}_X)Y} \cdot Y} + \bb{E}\sbr{\widetilde{\bb{E}}\sbr{D_{xm}^2 G(X,\s{L}_X,\tilde X)\tilde Y} \cdot Y} \geq 0,
\end{equation*}
for any $X,Y\in\bb{H}$ (where $\tilde Z$ stands for an independent copy of a random variable $Z\in\bb{H}$).
This in this case can be written
\begin{equation*}
	h(\s{L}_X)\bb{E}\sbr{\del{D^2 f(X)Y} \cdot Y} + \bb{E}\sbr{Df(X)Y} \cdot \bb{E}\sbr{D_m h(\s{L}_X,X)Y} \geq 0
\end{equation*}
or simply
\begin{equation} \label{eq:DX contradiction}
	\del{2 + \bb{E}\sbr{\sin(X)}}\bb{E}\sbr{e^{-X}Y^2} - \bb{E}\sbr{e^{-X}Y}\bb{E}\sbr{\cos(X)Y} \geq 0.
\end{equation}
For some $n \in \bb{N}$ to be determined, let $X$ be a random variable defined by $\bb{P}(X = 2n\pi) = \bb{P}(X = -2n\pi) = 1/2$, and let $Y = e^{X/2}$.
Note that $\sin(X) = 0$, $\cos(X) = 1$, and $e^{-X}Y^2 = 1$ a.s., and
\begin{equation*}
	\bb{E}\sbr{e^{-X}Y} = \bb{E}\sbr{e^{-X/2}} = \frac{1}{2}\del{e^{n\pi} + e^{-n\pi}} = \bb{E}\sbr{Y} = \cosh(n\pi).
\end{equation*}
Then \eqref{eq:DX contradiction} becomes
\begin{equation} \label{eq:DX contradiction1}
	2 - \cosh(n\pi) \geq 0.
\end{equation}
By choosing $n$ large enough, we derive a contradiction.
Thus \eqref{D} does not hold.

To show that \eqref{LL} does not hold is even more straightforward.
For this recall from \cite{gangbo2021mean,meszaros2021mean} that \eqref{LL} is equivalent to
\begin{equation}
	\bb{E}\sbr{\widetilde{\bb{E}}\sbr{D_{xm}^2 G(X,\s{L}_X,\tilde X)\tilde Y} \cdot Y} \geq 0, \ \ \forall X,Y\in\bb{H},
\end{equation}
which now reduces to
\begin{equation} \label{eq:LLX contradiction}
	- \bb{E}\sbr{e^{-X}Y}\bb{E}\sbr{\cos(X)Y} \geq 0.
\end{equation}
We may use the same example as above, or indeed simpler examples such as $X = 0$ and $Y = 1$, to contradict \eqref{eq:LLX contradiction}.

\end{proof}

\subsection{\eqref{X} does not imply \eqref{sigma}}

Definition \ref{def:X} makes \eqref{X} a rather strong condition, since it is supposed to be unconditional on the initial measure $m_0$ and the time horizon $T$.
It is unclear whether this necessarily implies the condition \eqref{sigma}.
In this subsection, we will provide an example for which \eqref{X} does hold, but only if we restrict to initial measures $m_0$ having uniformly bounded second moment.
That is, we assume that there exists a constant $M > 0$ such that $\int_{\R^d} \abs{x}^2 \dif m_0(x) \leq M^2$ for all initial measures $m_0$ in consideration.
Equivalently, this means that $\bb{E}[\, \abs{X_0}^2] \leq M^2$ for every initial random variable appearing in the definition of $\tilde E_T$ \eqref{eq:hat E1 1}.

\begin{proposition}
Take $d = k = 1$ and let $M>0$.
Let $G:\R\times\sP_2(\R)\to\R$ defined as $G(x,m) = g(x,\sigma_T(m))$, with $g:\R\times\R\to\R$, $\sigma_T:\sP_2(\R)\to\R$ and $\psi:\R\to\R$ be  given as $g(x,\sigma) = \frac{1}{2}x^2 \sigma$, $\sigma_T(m) =\int_{\R} \psi \dif m$, with $\psi$ continuously differentiable satisfying $\psi'(x)^2 \leq \frac{1}{M^2}\psi(x)$ and $\psi(x) \geq 1$, for $x\in\R$. Then, when restricted to the set $\left\{X\in\bb{H}: \bb{E}[\, \abs{X_0}^2] \leq M^2\right\}$, $\tilde E_T$ satisfies \eqref{X} strictly. However, in general neither \eqref{sigma} nor -\eqref{sigma} holds true on the set of measures $\left\{m\in\sP_2(\R): \int_{\R^d} \abs{x}^2 \dif m(x) \leq M^2\right\}.$
\end{proposition}

\begin{proof}
We notice that \eqref{eq:hat E1 1} becomes
\begin{equation} \label{eq:hat E1 2}
	\tilde E_T(X) = \frac{X_0}{1 + T\bb{E}[\psi(X)]},
\end{equation}
and \eqref{eq:E2 1}  becomes
\begin{equation} \label{eq:E2 2}
	\Omega(\sigma) = \int_{\R} \psi\del{\frac{x}{1+T\sigma}}\dif m_0(x).
\end{equation}
Differentiate \eqref{eq:hat E1 2} to get
\begin{equation}
	D\tilde E_T(X)Y = -\frac{X_0}{\del{1 + T\bb{E}[\psi(X)]}^2}T\bb{E}[\psi'(X)Y].
\end{equation}
We use the assumption $\bb{E}[X_0^2] \leq M^2$ to get
\begin{equation}
	\bb{E}\sbr{\del{D\tilde E_T(X)Y}Y} = -\frac{T\bb{E}[X_0Y]\bb{E}[\psi'(X)Y]}{\del{1 + T\bb{E}[\psi(X)]}^2}
	\leq \frac{TM\bb{E}[\psi'(X)^2]^{1/2}}{\del{1 + T\bb{E}[\psi(X)]}^2}\bb{E}[Y^2].
\end{equation}
We now use the fact that $\psi'(x)^2 \leq \frac{1}{M^2}\psi(x)$ and $\psi(x) \geq 1$ to deduce
\begin{equation}
	\bb{E}\sbr{\del{D\tilde E_T(X)Y}Y}
	\leq \frac{T\bb{E}[\psi(X)]^{1/2}}{\del{1 + T\bb{E}[\psi(X)]}^2}\bb{E}[Y^2]
	< \bb{E}[Y^2].
\end{equation}
It follows that \eqref{X} holds (strictly).
Note that this is unconditional on $T$, and the only condition on $m_0$ is the moment condition $\int_{\R} x^2 \dif m_0(x) \leq M^2$.

We now show that \eqref{sigma} need not hold in general. For this, consider $\psi(x) = 1 + \del{\frac{x}{2M} - a}^2$, where $a\in\R$ is a given parameter. It is immediate to see that $\psi$ satisfies the assumptions.
We will take $m_0 = \delta_M$, a Dirac mass concentrated at $M$.
We rewrite \eqref{eq:E2 2} and to get
\begin{equation*}
	\Omega(\sigma) = 1 + \del{\frac{1}{2(1+T\sigma)} - a}^2 \quad \Rightarrow \quad \Omega'(\sigma) = \del{a- \frac{1}{2(1+T\sigma)}}\frac{T}{2(1+T\sigma)^2}.
\end{equation*}
Direct computation yields that when restricted to the set of measures having second moments uniformly bounded by $M^2$, $\range\sigma_T=\left[1,\max\{1+a^2,1+(a-1/2)^2\}\right]$.

Suppose that $|a|\ge|a-1/2|$, then we find that $\lim_{\sigma \to 1+a^2} \Omega'(\sigma) < 1$, so for values of $\sigma$ sufficiently close to $1+a^2$, we have $\Omega'(\sigma) < 1$.

Similarly, if $|a|<|a-1/2|$, we find that $\lim_{\sigma \to 1+(a-1/2)^2} \Omega'(\sigma) < 1$, so for values of $\sigma$ sufficiently close to $1+(a-1/2)^2$, we have $\Omega'(\sigma) < 1$

This shows that the monotonicity condition \eqref{sigma} fails.

To show that -\eqref{sigma} fails, we need to impose a condition on $a$ and $T$, e.g.
	\begin{equation}
		a > \frac{2(1+T)^2}{T} + \frac{1}{2(1+T)}.
	\end{equation}
	Then it follows that $\Omega'(\sigma) > 1$ for $\sigma \geq 1$ close enough to 1.
	From here we indeed see that -\eqref{sigma} does not hold.
\end{proof}

\subsection{\eqref{sigma} does not imply \eqref{LL}, \eqref{D}, or \eqref{X}} \label{sec:sigma does not imply others}
In this subsection we restrict our attention to a special class of data $G$. First, we have the following result.
\begin{lemma}\label{lem:sigmanot}
Consider $G:\R^d\times\sP_2(\R^d)\to\R$ to have the form $G(x,m) = x \cdot \sigma_T(m)$, where $\sigma_T:\sr{P}_2(\R^d) \to \bb{R}^d$.
Conditions \eqref{LL}, \eqref{D} and \eqref{X} are equivalent.
\end{lemma}

\begin{remark}
	In Lemma \ref{lem:sigmanot}, it is crucial that \eqref{X} be required to hold independently of the time horizon $T$.
	Otherwise, the equivalence may not hold.
\end{remark}

\begin{proof}
By definition it is immediate to see that \eqref{LL} and \eqref{D} are both equivalent to the condition
\begin{equation} \label{eq:LLD1}
	\del{\sigma_T(m_1) - \sigma_T(m_2)} \cdot\int_{\bb{R}^d} x \dif\, (m_1 - m_2)(x) \geq 0 \quad \forall m_1,m_2\in\sP_2(\R^d).
\end{equation}
Now let us consider condition \eqref{X}.
The optimal trajectory starting at $x$ finishes at $y = x - T\sigma_T(m)$, so that $E_T(m)$ is a single-valued function
\begin{equation*}
	E_T(m) = (x \mapsto x - T\sigma_T(m))_\sharp m_0.
\end{equation*}
The lifted version is
\begin{equation*}
	\tilde E_T(X) = X_0 - T\sigma_T(\s{L}_X),
\end{equation*}
where $X_0$ is any random variable whose law is $m_0$.
Condition \eqref{X} becomes
\begin{equation} \label{eq:X1}
	-T \del{\sigma_T(\s{L}_{X_1}) - \sigma_T(\s{L}_{X_2})} \cdot\bb{E}\sbr{X_1 - X_2} \leq \bb{E}\abs{X_1 - X_2}^2 \quad \forall X_1, X_2 \in \bb{H}.
\end{equation}
For \eqref{eq:X1} to hold independently of $T$, as required by condition \eqref{X}, we must have
\begin{equation}
	\del{\sigma_T(\s{L}_{X_1}) - \sigma_T(\s{L}_{X_2})}\cdot \bb{E}\sbr{X_1 - X_2} \geq 0 \quad \forall X_1, X_2 \in \bb{H},
\end{equation}
which is the same as \eqref{eq:LLD1}.
\end{proof}

\begin{proposition} \label{pr:sigma not LLDX}
Let $G$ be as in Lemma \ref{lem:sigmanot}. Then, in general \eqref{sigma} does not imply any of the conditions \eqref{LL}, \eqref{D} or \eqref{X}.
\end{proposition}

\begin{proof}
For concreteness, we will assume that $\sigma_T(m) = \int_{\bb{R}^d} \psi \dif m$ for some $\psi:\bb{R}^d \to \bb{R}^d$.
Then \eqref{eq:LLD1} becomes
\begin{equation} \label{eq:LLD2}
	\int_{\bb{R}^d} \psi(x) \dif\,(m_1 - m_2)(x) \cdot \int_{\bb{R}^d} x \dif\, (m_1 - m_2)(x) \geq 0 \quad \forall m_1,m_2\in\sP_2(\R^d).
\end{equation}

Let us now examine the condition \eqref{sigma}.
To see this, write $g(x,\sigma) = x \cdot \sigma$ so that $G(x,m) = g(x,\sigma_T(m))$.
Observe that
\begin{equation}
	\Omega(\sigma) = \sigma_T\del{(x \mapsto x - T\sigma)_\sharp m_0} = \int_{\bb{R}^d} \psi(x-T\sigma)\dif m_0(x).
\end{equation}
Equation \eqref{sigma} is now equivalent to
\begin{equation} \label{eq:sigma1}
	\int_{\bb{R}^d} \del{\psi(x-T\sigma_1) - \psi(x-T\sigma_2)} \cdot (\sigma_1 - \sigma_2)\dif m_0(x)
	\leq \abs{\sigma_1 - \sigma_2}^2 \quad \forall \sigma_1,\sigma_2 \in \bb{R}^d.
\end{equation}
For \eqref{eq:sigma1} to hold without conditions on $m_0$ and $T$, it is necessary and sufficient for $\psi$ to be a monotone vector field, i.e.
\begin{equation} \label{eq:psi monotone}
	\del{\psi(\sigma_1) - \psi(\sigma_2)} \cdot (\sigma_1 - \sigma_2) \geq 0 \quad \forall \sigma_1,\sigma_2 \in \bb{R}^d.
\end{equation}
To see that \eqref{eq:psi monotone} is necessary, in \eqref{eq:sigma1} let $m_0 = \delta_0$ be a Dirac mass at the origin, replace $\sigma_i$ with $-\frac{1}{T}\sigma_i$ for $i = 1,2$  and rearrange to get
\begin{equation}
	\del{\psi(\sigma_1) - \psi(\sigma_2)} \cdot (\sigma_1 - \sigma_2) \geq -\frac{1}{T}\abs{\sigma_1 - \sigma_2}^2.
\end{equation}
Then let $T \to \infty$ to get \eqref{eq:psi monotone}.
Conversely, if \eqref{eq:psi monotone} holds, then for arbitrary $T$ and $m_0$ we have
\begin{equation}
	\int_{\bb{R}^d} \del{\psi(x-T\sigma_1) - \psi(x-T\sigma_2)} \cdot (\sigma_1 - \sigma_2)\dif m_0(x)
	\leq 0 \quad \forall \sigma_1,\sigma_2 \in \bb{R}^d,
\end{equation}
which implies \eqref{eq:sigma1}.

We can now construct a simple example to show that the condition \eqref{sigma} may be satisfied even when none of the conditions \eqref{LL}, \eqref{D}, or \eqref{X} are satisfied.
For simplicity we assume the dimension $d = 1$.
It suffices to take any increasing differentiable function $\psi$ such that $\psi'$ is not constant.
In this case the monotonicity of $\psi$ implies \eqref{sigma}.
On the other hand, \eqref{LL}, \eqref{D}, and \eqref{X} are all equivalent to \eqref{eq:LLD2}, which in turn implies
\begin{equation}
	\bb{E}\sbr{\psi'(X)Y}\bb{E}[Y] \geq 0 \quad \forall X,Y \in \bb{H}.
\end{equation}
Let $x_1,x_2 \in \bb{R}$ satisfy $\psi'(x_1) \neq \psi'(x_2)$ and $\psi'(x_1) \neq 0$, and let $(X,Y)$ be equal to $(x_1,-\psi'(x_1)-\psi'(x_2))$ with probability $1/2$ and $(x_2,2\psi'(x_1))$ with probability $1/2$.
Then
\begin{equation}
	\begin{split}
		\bb{E}\sbr{\psi'(X)Y}\bb{E}[Y] &= \frac{1}{4}\del{-\psi'(x_1)^2 -\psi'(x_1)\psi'(x_2) + 2\psi'(x_2)\psi'(x_1)}\del{\psi'(x_1) - \psi'(x_2)}\\
		&= -\frac{1}{4}\psi'(x_1)\del{\psi'(x_1) - \psi'(x_2)}^2 < 0.
	\end{split}
\end{equation}
It follows that none of the conditions \eqref{LL}, \eqref{D}, or \eqref{X} are satisfied.

\end{proof}

\appendix

\section{Non-uniqueness for an optimal control problem} \label{sec:nonunique optimizers}

In this Section, we restate and then prove Theorem \ref{thm:nonunique optimizers}.
\begin{theorem}
	Let $\phi:\bb{R}^d \to \bb{R}$ be a continuous function.
	Assume that $\phi$ is not convex, non-constant, bounded below and has sub-quadratic growth at infinity, i.e.~there exist $C>0$ and $\alpha \in (0,1)$ such that $\phi(x)\le C(1+|x|^{1+\alpha})$ for all $x\in\R^d.$
	Then there exists $t^* > 0$ such that for every $t \geq t^*$, there exists $x \in \bb{R}^d$ such that $\frac{\abs{x-y}^2}{2t} + \phi(y)$ has at least two distinct minimizers.
\end{theorem}	

\begin{proof}
	\firststep
	In this step, we show that there exists $t^* > 0$ such that if $t \geq t^*$, then the function $\frac{\abs{x}^2}{2t} + \phi(x)$ is not convex.
	Since $\phi$ is not convex, there exist $x_0, h \in \bb{R}^d$ such that
	\begin{equation}
		\phi(x_0) > \frac{1}{2}\del{\phi(x_0 + h) + \phi(x_0 - h)}.
	\end{equation}
	Let $x \in \bb{R}^d$.
	If $\psi(x) = \frac{\abs{x}^2}{2t} + \phi(x)$, then
	\begin{equation*}
		\psi(x_0) - \frac{1}{2}\del{\psi(x_0 + h) + \psi(x_0 - h)} = \phi(x_0)- \frac{1}{2}\del{\phi(x_0 + h) + \phi(x_0 - h)} - \frac{\abs{h}^2}{2t}.
	\end{equation*}
	Taking $t^*$ large enough, we see that the right-hand side is positive for any $t \geq t^*$, which means that $\psi$ is not convex.
	
	\nextstep From now one we fix $t \geq t^*$ and set $\psi(x) = \frac{\abs{x}^2}{2t} + \phi(x)$.
	Note that
	\begin{equation*}
		\frac{\abs{x-y}^2}{2t} + \phi(y) = \frac{\abs{x}^2}{2t} - \frac{1}{t}x \cdot y + \psi(y).
	\end{equation*}
	To show that $\frac{\abs{x-y}^2}{2t} + \phi(y)$ has at least two distinct minimizers, it is enough to show that the same is true of $- \frac{1}{t}x \cdot y + \psi(y)$.
	Thus, to reach the desired conclusion, it is enough to prove the following claim:
	\begin{quote}
		There exists $a \in \bb{R}^d$ such that $a \cdot x + \psi(x)$ has at least two distinct minimizers.
	\end{quote}
	Recall that $\psi$ is not convex, so there exist $x_0,h \in \bb{R}^d$ (the ones found before) such that
	\begin{equation*}
		\psi(x_0) > \frac{1}{2}\del{\psi(x_0 + h) + \psi(x_0 - h)}.
	\end{equation*}
	Define
	\begin{equation}
		\tilde\psi(x) := \psi(x + x_0) - \psi(x_0) - b \cdot x, \quad b := \frac{\psi(x_0 + h) - \psi(x_0 - h)}{2\abs{h}^2}h.
	\end{equation}
	Then $\tilde \psi(0) = 0$ and $\tilde \psi(h) = \tilde{\psi}(-h) < 0$.
	Note that $x$ is a minimizer of $a\cdot x + \psi(x)$ if and only if $x-x_0$ is a minimizer of $(a+b)\cdot x + \tilde\psi(x)$, so it is enough to prove the claim with $\psi$ replaced with $\tilde \psi$.
	Thus, without loss of generality, $\psi(0) = 0$ and $\psi(-h) = \psi(h) < 0$ for some $h \in \bb{R}^d$.
	
	\nextstep Define $F:\bb{R}^d \to \bb{R}^d$ as follows.
	For each $i = 1,\ldots,d$ and $a\in\R^d$ let
	\begin{equation*}
		F_i^{\pm}(a) = \min\cbr{a \cdot x + \psi(x) : \pm x_i \geq 0}, \quad F_i(a) = F_i^{+}(a) - F_i^-(a).
	\end{equation*}
	Note that $F_i^{\pm}$ is well-defined because $a \cdot x + \psi(x) \to \infty$ as $\abs{x} \to \infty$, so it suffices to search for minimizers on a compact set.
	It is straightforward to see that $F$ is continuous.
	We now claim that it is coercive, i.e.
	\begin{equation} \label{eq:F coercive}
		\lim_{\abs{a} \to \infty} \frac{F(a) \cdot a}{\abs{a}} = \infty.
	\end{equation}
	For any $x \in \bb{R}^d$ we will use the notation $x_{-i}$ to mean the vector in $\bb{R}^{d-1}$ obtained by removing the $i$th coordinate from $x$.
	If $a_i \geq 0$, then $a \cdot x \geq a_{-i} \cdot x_{-i}$ whenever $x_i \geq 0$ and so
	\begin{equation*}
		F_i^{+}(a) \geq \inf \phi + \min\cbr{a_{-i} \cdot x_{-i} + \frac{\abs{x_{-i}}^2}{2t} : x_{-i} \in \bb{R}^{d-1}} = \inf \phi - \frac{t\abs{a_{-i}}^2}{2}.
	\end{equation*}
	On the other hand, by setting $x_{-i} = -ta_{-i}$ and $x_i = -s\abs{a}^\alpha$, with $s > 0$ to be determined, we get the estimate
	\begin{equation*}
		F_i^-(a) \leq C \cdot 2^\alpha(1+t^{1+\alpha}|a|^{1+\alpha} + s^{1+\alpha}|a|^{(1+\alpha)\alpha}) - s\abs{a}^\alpha a_i + \frac{s^2\abs{a}^{2\alpha}}{2t} - \frac{t\abs{a_{-i}}^2}{2}.
	\end{equation*}
	Subtracting these two inequalities, we get
	\begin{equation*}
		F_i(a) \geq \inf \phi - C \cdot 2^\alpha(1+t^{1+\alpha}|a|^{1+\alpha} + s^{1+\alpha}|a|^{(1+\alpha)\alpha}) - \frac{s^2}{2t}\abs{a}^{2\alpha} + s\abs{a}^{\alpha}a_i.
	\end{equation*}
	If $a_i \leq 0$, then apply the mirror image of this argument to get
	\begin{equation*}
		F_i(a) \leq -\inf \phi + C \cdot 2^\alpha(1+t^{1+\alpha}|a|^{1+\alpha} + s^{1+\alpha}|a|^{(1+\alpha)\alpha}) - \frac{s^2}{2t}\abs{a}^{2\alpha} + s\abs{a}^{\alpha}a_i.
	\end{equation*}
	In either case, we can conclude that
	\begin{equation*}
		\begin{split}
			F_i(a)a_i &\geq s|a|^\alpha a_i^2 -\del{-\inf \phi + C \cdot 2^\alpha(1+t^{1+\alpha}|a|^{1+\alpha} + s^{1+\alpha}|a|^{(1+\alpha)\alpha}) + \frac{s^2}{2t}\abs{a}^{2\alpha}}\abs{a_i}\\
		\Rightarrow
		F(a)\cdot a &\geq s\abs{a}^{2+\alpha}  - \del{-\inf \phi + C \cdot 2^\alpha(1+t^{1+\alpha}|a|^{1+\alpha} + s^{1+\alpha}|a|^{(1+\alpha)\alpha}) + \frac{s^2}{2t}\abs{a}^{2\alpha}}d^{1/2}\abs{a}\\
		&= \del{s - C \cdot 2^\alpha t^{1+\alpha}d^{1/2}}\abs{a}^{2+\alpha}  - \del{-\inf \phi + C \cdot 2^\alpha(1+s^{1+\alpha}|a|^{(1+\alpha)\alpha}) + \frac{s^2}{2t}\abs{a}^{2\alpha}}d^{1/2}\abs{a}\\
		\Rightarrow
		\frac{F(a)\cdot a}{|a|} &\geq \del{s - C \cdot 2^\alpha t^{1+\alpha}d^{1/2}}\abs{a}^{1+\alpha}  - \del{-\inf \phi + C \cdot 2^\alpha(1+s^{1+\alpha}|a|^{(1+\alpha)\alpha}) + \frac{s^2}{2t}\abs{a}^{2\alpha}}d^{1/2}.
		\end{split}
	\end{equation*}
	Choosing some $s > Ct^{1+\alpha}d^{1/2}$, and noticing that $1+\alpha > \max\cbr{2\alpha,(1+\alpha)\alpha}$, we then deduce \eqref{eq:F coercive}.
	
	\nextstep As each coordinate function $F_i$ is non-decreasing in the $a_i$ variable, the map $F$ is monotone. As it is also continuous and coercive, by the Browder--Minty Theorem (see for instance \cite[Theorem 3]{browder-minty}), there exists an $a \in \bb{R}^d$ such that $F(a) = 0$, hence $F_i(a) = F_i^-(a)$ for each $i = 1,\ldots,d$.
	We can conclude that all of these values are in fact the minimum of $a \cdot x + \psi(x)$, and that there exist minimizers $x^{\pm,i}$ such that $\pm x^{\pm,i}_i \geq 0$.
	If any two of these $2d$ minimizers are distinct, we are done.
	Suppose they are all identical.
	In this case they must all be 0 (every coordinate must be both non-negative and non-positive).
	Thus the minimum of $a \cdot x + \psi(x)$ is attained at 0 and is therefore equal to 0 (as $\psi(0)=0$).
	But notice that, since $\psi(h) = \psi(-h) < 0$, we must have either $a \cdot h + \psi(h) < 0$ or $a \cdot (-h) + \psi(-h) < 0$.
	This is a contradiction.
	Therefore, $a \cdot x + \psi(x)$ must have at least two distinct minimizers, which is what we needed to show.
\end{proof}
	
\noindent {\bf Acknowledgements.} PJG acknowledges the support of the National Science Foundation through NSF Grants DMS-2045027 and DMS-1905449. We acknowledge the support of the Heilbronn Institute for Mathematical Research and the UKRI/EPSRC Additional Funding Programme for Mathematical Sciences through the focused research grant ``The master equation in Mean Field Games''. ARM has also been partially supported by the EPSRC via the NIA with grant number EP/X020320/1	 and by the King Abdullah University of Science and Technology Research Funding (KRF) under Award No. ORA-2021-CRG10-4674.2. Last but not least, we would like to thank the referee for their careful reading of our manuscript and for the important comments they gave.
	
\medskip

\noindent {\bf Conflict of interest.} The authors declare that they do not have any conflicts of interests.	

\medskip

\noindent {\bf Data Availability Statement.} Data sharing not applicable to this article as no datasets were generated or analyzed during the current study.	
	
\bibliographystyle{plain}	
\bibliography{mybib}

\end{document}